\documentclass
[11pt]
{amsart}
\usepackage{amsmath,amsfonts,amssymb,amsthm}
\usepackage{amscd}
\usepackage{latexsym}
\usepackage[usenames]{color}
\usepackage{mathrsfs}

\date{}

\usepackage[left=1in,right=1in,top=1in,bottom=1in]{geometry}

\input xy  \xyoption{all}
\newtheorem{deff}{Definition}[section]
\newtheorem{lemma}[deff]{Lemma}
\newtheorem{theorem}[deff]{Theorem}
\newtheorem{corollary}[deff]{Corollary}

\newtheorem{proposition}[deff]{Proposition}
\newtheorem{fact}[deff]{Fact}
\newtheorem{em-example}[deff]{Example}
\newtheorem{em-def}[deff]{Definition}        
\newtheorem{em-remark}[deff]{Remark}         
\newtheorem{em-question}[deff]{Question}

\newtheorem{claim}{Claim}

\newenvironment{example}{\begin{em-example} \em }{ \end{em-example}}
\newenvironment{definition}{\begin{em-def} \em  }{ \end{em-def}}
\newenvironment{remark}{\begin{em-remark} \em }{\end{em-remark}}
\newenvironment{question}{\begin{em-question}\em }{\end{em-question}}

\newcommand\DEQ{\hfill $\bigtriangleup$ \medskip}

\def\ker{\mathop{\rm ker}}

\def\:{\nobreak \hskip .1111em\mathpunct {}\nonscript \mkern
-\thinmuskip {:}\hskip .3333emplus.0555em\relax}

\catcode`\@=12
\def\T{{\mathbb T}}

\def\Z{{\mathbb Z}}
\def\N{{\mathbb N}}

\def\R{{\mathbb R}}
\def\Q{{\mathbb Q}}

\def\L{{\mathcal L}}

\def\cont{\mathfrak c}

\newcommand{\cc}{countably\ compact}

\def\cont{\mathfrak{c}}

\def\Zc{$\mathscr{Z}_c$}
\def\Zgen{$\mathscr{Z}$}
\def\Zm{$\mathscr{Z}_m$}
\def\Zcm{$\mathscr{Z}_{cm}$}

\def\Ptd{$\mathscr{E}_{td}$}
\def\Pt{$\mathscr{E}_{t}$}
\def\Pmin{$\mathscr{E}_{min}$}
\def\Avoid#1{$\mathscr{A}_{{#1}}$}

\def\pt#1{#1}
\def\jn#1{{\sl #1}}

\title[Characterizing 
Lie groups]{Characterizing Lie groups by controlling their zero-dimensional subgroups}

\begin{document}

\author[D. Dikranjan]{Dikran Dikranjan}
\address[D. Dikranjan]{Dipartimento di Matematica e Informatica\\
Universit\`{a} di Udine\\
Via delle Scienze  206, 33100 Udine\\
Italy}
\email{dikran.dikranjan@uniud.it}

\author[D. Shakhmatov]{Dmitri Shakhmatov}
\address[D. Shakhmatov]{Division of Mathematics, Physics and Earth Sciences\\
Graduate School of Science and Engineering\\
Ehime University\\
Matsuyama 790-8577\\
Japan}
\email{dmitri.shakhmatov@ehime-u.ac.jp}



\thanks{The second named author was partially supported by the Grant-in-Aid for
Scientific Research~(C) No.~22540089 by the Japan Society for the Promotion of Science (JSPS)}

\maketitle
\begin{abstract}
We provide characterizations of Lie groups as compact-like groups in which all closed zero-dimensional metric (compact) subgroups are discrete. The ``compact-like'' properties we consider include (local) compactness, (local) $\omega$-boundedness, (local) countable compactness, (local) precompactness, (local)
minimality  and sequential completeness. Below is a sample of our characterizations:

(i)~A topological group is a Lie group if and only if it is locally compact and has no infinite compact metric zero-dimensional subgroups.
 
(ii)~An abelian topological group $G$ is a Lie group if and only if $G$ is  locally minimal, locally precompact and all closed metric zero-dimensional subgroups of $G$ are discrete.

(iii)~An abelian topological group is a compact Lie group if and only if it is minimal and has no infinite closed metric zero-dimensional subgroups.

(iv)~ An infinite topological group is a compact Lie group if and only if it is 
 sequentially complete, precompact, locally  minimal, 
contains 
a non-empty open
connected subset 
and all its compact 
metric zero-dimensional subgroups are finite. 
\end{abstract}

{\em All topological groups considered in this paper are assumed to be Hausdorff.\/}

As usual, $\cont$ denotes the cardinality of the continuum.
 
Recall that a topological space $X$ is:
\begin{itemize}
\item
{\em $\omega$-bounded\/} if the closure of every countable subset of $X$ is compact,
\item
{\em countably compact\/} if every countable open cover of $X$ has a finite subcover,
\item
{\em pseudocompact\/} if every real-valued continuous function defined on $X$ is bounded.
\item
{\em zero-dimensional\/} if $X$ has a base consisting of sets that are simultaneously open and closed in $X$.  
\end{itemize}

A topological group $G$ is {\em locally $\omega$-bounded\/} ({\em locally countably compact\/}, {\em locally pseudocompact\/})
if it has an open neighbourhood $U$ of its identity element whose closure $\overline{U}$ is $\omega$-bounded (countably compact, pseudocompact, respectively). We say that a topological group is {\em (locally) precompact\/} if its two-sided uniformity completion is (locally) compact. Recall that a topological group $G$ is called \emph{sequentially complete\/} if every Cauchy sequence in $G$ with respect to the two-sided uniformity of $G$ converges to some element of $G$ \cite{DT1,DT2}. For metrizable groups, sequential completeness coincides with completeness, while pseudocompactness coincides with compactness. 

Recall that a  topological group $G$ is called \emph{minimal\/} if every continuous isomorphism $f:G\to H$, where $H$ is Hausdorff  topological group,  is a topological isomorphism  \cite{S}. 
Clearly, all compact groups are  minimal. 
Easy examples show that  the converse does not  hold in general.  Nevertheless, a somewhat weaker implication holds in the case of abelian groups:

\begin{fact}
\label{precompactness:theorem:for:minimal:groups}
{\rm (Prodanov and Stoyanov; see \cite{DPS})} A minimal abelian group $G$ is precompact.
\end{fact}

A common generalization of locally compact groups and minimal groups was proposed by Morris and Pestov.   A topological group $(G, \tau)$ is {\em locally minimal} if there exists a neighborhood $V$ of $e_G$  such that whenever $\sigma\subseteq \tau$  is a Hausdorff group topology on $G$ such that $V$ is a $\sigma$-neighborhood of  $e_G$,  then $\sigma=\tau$ \cite{MP}. Besides locally compact groups, the class of locally minimal groups contains also all additive groups of normed vector topological spaces \cite{ACDD,MP}.  We refer the reader to \cite{ACDD,ACDD2,DM} for some recent progress in this area. 

The following diagram summarizes the connections between the compactness-like properties considered in this paper. 

\medskip
\begin{center}
${\xymatrix@!0@C5.0cm@R=1.2cm{
\mbox{minimal} \ar@{->}[r] & \mbox{locally minimal} & \\
\mbox{compact}  \ar@{->}[r] \ar@{->}[d] \ar@{->}[u] & \mbox{locally compact} \ar@{->}[r]  \ar@{->}[d] \ar@{->}[u] & \mbox{complete} \ar@{->}[dd] \\
\mbox{$\omega$-bounded}  \ar@{->}[r] \ar@{->}[d] & \mbox{locally $\omega$-bounded}\ar@{->}[d]   &  \\
\mbox{countably compact}  \ar@{->}[r] \ar@{->}[d] & \mbox{locally countably compact}\ar@{->}[d] \ar@{->}[r]  & \mbox{sequentially complete} \\
\mbox{pseudocompact}  \ar@{->}[r] \ar@{->}[d] & \mbox{locally pseudocompact}\ar@{->}[d] & \\
\mbox{precompact}  \ar@{->}[r] & \mbox{locally precompact} & \\
   }}$
\end{center}
\medskip
\begin{center}
Diagram 1.
\end{center}
\medskip

None of the arrows in this diagram is reversible. For each one of the properties $\mathfrak P$ in the first column, except 
``minimal''
(i.e., for $\mathfrak P\in \{$compact,  $\omega$-bounded, countably compact, pseudocompact, precompact$\}$), 
$\mathfrak P$ is equivalent to ``precompact and locally $\mathfrak P$''.

\section{Exotic tori and exotic Lie groups}\label{Exotic}

A topological group $G$ is called an {\em NSS group\/} (an abbreviation for No Small Subgroups) if a suitable neighborhood of its identity element contains only the trivial subgroup. 

A Lie group is a set endowed with compatible structures of a group and of an analytic manifold over the field of real numbers.  The celebrated Hilbert's fifth problem was asking whether the condition ``analytic manifold'' can be relaxed to ``topological manifold''.   An essential ingredient in the positive solution of the problem was the following criterion obtained by Glushkov \cite{G}: {\em A locally compact group  is a Lie group if and only if it is an NSS group\/}. 

The following definition is due to Dikranjan and Prodanov \cite{DP2}.

\begin{definition}
\label{four:properties}
Let 
$\mathscr{C}$ be a class of topological groups. 
For a topological group $G$, consider
the following properties: 
\begin{itemize}
\item[(\Ptd)] for every closed subgroup  $H$ of $G$, the torsion subgroup $t(H)$ of $H$ is dense in $H$;
\item[(\Pt)] every closed non-trivial subgroup of $G$ contains non-trivial torsion elements; 
\item[(\Pmin)]
every closed non-trivial subgroup of $G$ contains a minimal (with respect to the set inclusion) closed non-trivial subgroup of $G$;
\item[(\Avoid{\mathscr{C}})] $G$ contains no subgroups topologically isomorphic to any member of $\mathscr{C}$.
\end{itemize}
\end{definition}

As was demonstrated in \cite{DP2} and \cite{DSt}, the following two classes $\mathscr{C}$ are 
of special importance.
\begin{definition}
Define $\mathscr{P}=\{\mathbb{Z}_p:p\in \mathbb{P}\}$ and  $\mathscr{P}^*=\{\mathbb{Z}\}\cup \mathscr{P}$, where $\mathbb{P}$ is the set of all prime numbers, $\mathbb{Z}_p$ is the topological group of $p$-adic integers and $\Z$ is the group of integers with the discrete topology.
\end{definition}

A compact abelian Lie group $G$ has the form $G=\T^m \times F$, for some finite abelian group $F$, so $t(G) =(\Q/\Z)^m \times F$ is dense in $G$. 
Since every compact Lie group is covered by its compact abelian Lie subgroups, it follows that compact Lie groups satisfy property \Ptd. The implications 
\Ptd$\to$\Pt$\to$\Pmin\  are obvious. For every prime number $p$, the group $\Z_p$ of $p$-adic integers does not satisfy \Pmin, so  one has also the implication
\Pmin$\to$\Avoid{\mathscr{P}}. Therefore, 
\begin{equation}
\label{chains:of:Es}
\mbox{compact Lie} \to \mbox{\Ptd} \to \mbox{\Pt} \to \mbox{\Pmin} \to \mbox{\Avoid{\mathscr{P}}}.
\end{equation}
In this chain of implications,
the four properties from Definition \ref{four:properties} turned out to be equivalent for all compact groups:

\begin{theorem}
\label{exotic:tori}
{\rm (\cite{DSt}; earlier proved in \cite{DP2} in the abelian case)}
For every compact group, conditions 
\Ptd, \Pt, \Pmin\ 
and \Avoid{\mathscr{P}} are equivalent.
\end{theorem}

The compact abelian groups having the equivalent properties \Ptd, \Pt,  \Pmin\  and \Avoid{\mathscr{P}} were extensively studied in \cite{DP2}
under the name {\em exotic tori\/}.
In addition to the obvious connection to the usual tori (i.e., compact connected abelian Lie groups), exotic tori appear naturally in the study of minimal torsion abelian groups;
see the paragraph preceding Fact \ref{precompactness:theorem:for:minimal:groups} for definition.
The completion of a minimal torsion abelian group must be an exotic torus, 
and the  minimal torsion abelian groups are precisely the dense essential torsion subgroups of the exotic tori; see \cite{DPS}
for the definition of an essential subgroup. All exotic tori are finite-dimensional and have other nice properties (for more details see \cite{DP2} or \cite{DPS}).   An example of an exotic torus (i.e., a compact abelian group satisfying equivalent conditions from Theorem \ref{exotic:tori}) which is not a Lie group was exhibited in \cite{DP2}, so none of the conditions from Theorem \ref{exotic:tori} characterize Lie groups, even in the abelian case.
 
Theorem \ref{exotic:tori} cannot be extended to locally compact groups, or even arbitrary Lie groups, as $\R$ (as well as  the discrete group $\Z$) satisfies \Avoid{\mathscr{P}}, but satisfies neither \Pt\  nor \Pmin. In order to eliminate $\R$ and the discrete group $\Z$ one can impose in addition the condition \Avoid{\{\Z\}}\  which obviously follows from \Pmin.  According to a well known property of locally compact groups, a locally compact group $G$ satisfies \Avoid{\{\Z\}}\ precisely when $G$ is covered by its compact subgroups. So an appropriate locally compact version of Theorem \ref{exotic:tori} can be announced as follows: 

 \begin{theorem}\label{DSto}{\rm  \cite{DSt}} 
 For each locally compact group, conditions \Ptd, \Pt,  \Pmin\  and  \Avoid{\mathscr{P}^*}\ are equivalent.
\end{theorem}

The locally compact groups satisfying the equivalent conditions of Theorem \ref{DSto} were called {\em exotic Lie groups} and studied in detail \cite{DSt} and  \cite{Dtor}. It was shown, in particular, that exotic Lie groups are finite-dimensional. It is worth noticing that while all compact Lie groups are exotic Lie groups,  a non compact Lie group (e.g., $\R$) need not be an exotic Lie group. 

Let us note that \Pmin\  has an advantage over \Pt\  and \Avoid{\mathscr{C}}, since it is entirely formulated in terms of the lattice $\L(G)$ of closed subgroups of $G$. Indeed, \Pmin\  says that the poset $\L(G)$ is atomic, i.e., every element dominates a minimal element (an atom) of $\L(G)$. This lattice theoretic property was used in \cite{DP2} to deduce, via Pontryagin duality, that a compact abelian group $G$ satisfies \Pmin\  precisely when the lattice $\L(\widehat G)$ satisfies the dual property. Namely, every proper subgroup of the discrete group $\widehat G$ is contained in a maximal subgroup of $\widehat G$, or equivalently, $\widehat G$ has no divisible quotient groups. The abelian  groups with this property were described in \cite{DP2}.

\section{Three properties shared by all Lie groups}

In Theorem \ref{exotic:tori} we saw that the conditions \Ptd, \Pt, \Pmin\  and \Avoid{\mathscr{P}}\ are equivalent for compact groups and define a class of compact groups containing all compact Lie groups. On the other hand, the conditions \Ptd, \Pt,  \Pmin\  and  \Avoid{\mathscr{P}^*}\ are equivalent for locally compact groups according to Theorem \ref{DSto}, but in this case the class of locally compact groups determined by these equivalent conditions does not contain all Lie groups (e.g., it does not contain $\R$). In other words, the conditions considered in \cite{DP2,DSt} characterize classes of (locally) compact groups that {\em properly} contain the class of all compact Lie groups, but do not contain the class of all Lie groups. 

In this paper we skip the too restrictive condition \Avoid{\mathscr{P}^*} and we find an appropriate strengthening of the condition \Avoid{\mathscr{P}}\ 
that allows us to characterize Lie groups not only in the class of locally compact groups but also in some much wider classes of compact-like groups. 
A hint on how to achieve this comes from the last of the equivalent properties in  Theorem \ref{exotic:tori}, namely \Avoid{\mathscr{P}}. Indeed, one easily notices that the countable family of groups excluded by the property \Avoid{\mathscr{P}} consists exclusively of compact
zero-dimensional groups. Of course, an obvious way of strengthening the property \Avoid{\mathscr{P}} is to prohibit {\em all non-discrete closed zero-dimensional groups\/}. This leads us exactly to the property \Zgen\ in our next definition.

\begin{definition}
\label{Z:def}
For a topological group $G$, define the following conditions:
\begin{itemize}
\item[(\Zgen)] every closed zero-dimensional  subgroup of $G$ is discrete, 
\item[(\Zm)] every closed zero-dimensional metric subgroup of $G$ is discrete, 
\item[(\Zcm)] every compact metrizable zero-dimensional subgroup of $G$ is finite.
\end{itemize}
\end{definition}

All three conditions above strengthen \Avoid{\mathscr{P}}\ in an obvious way.

The property \Zgen\ has been well studied in functional analysis. 

\begin{remark}
\label{Banakh:remark}
The additive group of a Banach space $B$ has property \Zgen\ (equivalently, property \Zm) if and only if $B$ contains no subspace isomorphic to $c_0$; see \cite[Theorem 4.1]{ADG}. In particular, the additive group of the Hilbert space has property \Zgen\ (equivalently, property \Zm) \cite{DG}.
\end{remark}

The three properties from Definition \ref{Z:def} are ultimately related to Lie groups, as the following proposition shows:
\begin{proposition}
\label{arrows:between:Z}
For every topological group $G$, the following implications hold:
\begin{equation}
\label{Z:implications}
\mbox{Lie}\to \mbox{\Zgen}\to\mbox{\Zm}\to\mbox{\Zcm}.
\end{equation}
\end{proposition}

\begin{proof}
Assume that $G$ is a Lie group, and let $N$ be a closed zero-dimensional subgroup of $G$.  Being a closed subgroup of the Lie group $G$, $N$ is a Lie group itself. Since a zero-dimensional Lie group is discrete, $N$ must be discrete.  Thus, \Zgen\ holds. This finishes the proof of the first implication  
$Lie$ $\to\, $\Zgen. The implications \Zgen$\ \to\ $\Zm$\ \to\ $\Zcm\   are obvious.
\end{proof}

Note that the chain of implications in \eqref{Z:implications} is similar to that in \eqref{chains:of:Es}.

One may wonder why  Definition \ref{Z:def} omits the following  natural condition:

\begin{itemize}
\item[(\Zc)] every compact zero-dimensional subgroup of $G$ is finite.
\end{itemize}

It turns out that this property is equivalent to \Zcm. Indeed, the implication \Zc$\to$\Zcm\ is trivial, and the converse implication easily follows from the following fact.

\begin{fact}
\cite{HHM}
\label{compact:metrizable:subgroup}
Every infinite compact group has an infinite compact metrizable subgroup.
\end{fact}

Since closed subgroups of compact groups are compact, properties  \Zgen\ and \Zc\ are equivalent for compact groups. Since \Zc\ is equivalent to \Zcm, combining this with Proposition \ref{arrows:between:Z}, we conclude that the three properties \Zgen, \Zm\ and \Zcm\ from  Definition \ref{Z:def} are equivalent for compact groups.

In the rest of this section we collect examples showing that the implications in \eqref{Z:implications} are not reversible.

\begin{example}
\label{Hilbert:space}
{\em The additive group $H$ of an infinite-dimensional Hilbert space is locally minimal, (sequentlially) complete, satisfies property \Zgen\ yet is not a Lie group.\/}
Indeed, local minimality of $H$ follows from \cite{MP}. By Remark \ref{Banakh:remark}, $H$ has property \Zgen.
\end{example}

The following proposition outlines two trivial yet useful ways of constructing groups satisfying properties \Zm\ and \Zcm.

\begin{proposition}
\label{trivial:ways}
\begin{itemize}
\item[(i)] A topological group without non-trivial convergent sequences has property \Zm.
\item[(ii)] A topological group without infinite compact metric subgroups has property \Zcm.
\end{itemize}
\end{proposition}
\begin{proof}
Item (i) follows from the fact that every non-discrete metric space contains a non-trivial convergent sequence. Item (ii) is clear.
\end{proof}

\begin{corollary}
\label{non-discrete:group:without:convergent:sequences}
If $G$ is a non-discrete zero-dimensional group without non-trivial convergent sequences, then $G$ satisfies \Zm\ but does not satisfy \Zgen.
\end{corollary}
\begin{proof}
Indeed,  $G$ satisfies \Zm \ by Proposition \ref{trivial:ways}(i). Since $G$ is non-discrete and zero-dimensional, it does not satisfy \Zgen.
\end{proof}

For an abelian group $G$, we denote by $G^\#$ the group $G$ equipped with its {\em Bohr topology\/}, i.e., the strongest precompact group topology on $G$. 
The completion $bG$ of $G^\#$ is a compact abelian group called the {\em Bohr compactification\/} of $G$. (The terms  Bohr topology and Bohr compactification have been chosen as a reward to Harald Bohr for his work \cite{B} on almost periodic functions closely related to the Bohr compactification.) It is known that $G^\#$ is zero-dimensional \cite{Sh} and sequentially complete \cite{DT1,DT2}. 

\begin{corollary}
\label{Bohr:topology}
For every infinite abelian group $G$, the group $G^\#$ satisfies \Zm\ but does not satisfy  \Zgen. 
\end{corollary}
\begin{proof}
It is well known that $G^\#$ has no non-trivial convergent sequences \cite{Flor}; see also \cite{Glicksberg}. Clearly, $G^\#$ is precompact. Since it is also infinite, it cannot be discrete.
Now the conclusion follows from Corollary \ref{non-discrete:group:without:convergent:sequences}, as $G^\#$ is zero-dimensional.
\end{proof}

\begin{corollary}
\label{Bohr:proposition}
 For every infinite abelian group $G$  there exists a zero-dimensional sequentially complete precompact group  topology $\tau$ on $G$ such that $(G, \tau)$ satisfies \Zm\ but does not satisfy  \Zgen. 
\end{corollary}

\begin{proof} It suffices to take $\tau$ to be the Bohr topology on $G$, so that $G^\#=(G,\tau)$. 
\end{proof}

In connection with the last three corollaries, it may be worth noticing that a counter-example to the implication \Zm$\to$\Zgen\ {\em cannot\/} be metrizable.

\begin{corollary}
\label{small:groups:are:Zcm}
Every topological group $G$ with $|G|<\cont$  satisfies \Zcm.
\end{corollary}
\begin{proof} Recall that infinite compact (metric) groups have size $\geq \cont$. From this and the assumption of our corollary, we conclude that 
that all compact metric subgroups of $G$ are finite. Now $G$ satisfies \Zcm\ by Proposition \ref{trivial:ways}(ii). 
\end{proof}

\begin{corollary}
\label{groups:of:size:smaller:than:c}
Let $G$ be a non-discrete metric group such that $|G|<\cont$. Then $G$ satisfies \Zcm\ but does not have property \Zm.
\end{corollary}
\begin{proof}
By last corollary, $G$ has property \Zcm. It easily follows from $|G|<\cont$ that $G$ is zero-dimensional. Since $G$ is metric and non-discrete, it fails property \Zm.
\end{proof}

\begin{remark}
\label{TVS:remark}
(i) {\em If continuous homomorphisms from a topological group $G$ to $\R$ separate points of $G$, then $G$ contains no non-trivial compact subgroups; in particular, $G$ satisfies \Zcm\/} by Proposition \ref{trivial:ways}(ii). Indeed, assume that $K$ is a non-trivial compact subgroup of $G$. By our assumption, there exists a continuous homomorphism $f:G\to\R$ such that $f(K)$ is non-trivial. Since $f(K)$ is compact, this contradicts the fact that $\R$ has no non-trivial compact subgroups.

(ii) It follows from (i) that {\em the additive group $G$ of every topological vector space over $\R$ has no non-trivial compact subgroups; in particular, $G$ satisfies \Zcm\/}. 
\end{remark}

\begin{example}\label{example:Z}
(i)
{\em The complete metric group $\R^\N$  satisfies \Zcm\ but does not satisfy \Zm\/}. Indeed, since $\Z^\N$ is an infinite non-discrete closed zero-dimensional subgroup of $\R^\N$, the group $\R^\N$ fails property \Zm. The rest follows from Remark \ref{TVS:remark}(ii).

(ii)
{\em The additive group $G$ of the Banach space $c_0$ is a locally minimal (sequentially) complete metric abelian group satisfying \Zcm\ but failing \Zm\/}.
Indeed, $G$ satisfies \Zcm\ by Remark \ref{TVS:remark}(ii). On the other hand,  $G$ does not satisfy \Zm\ by Remark \ref{Banakh:remark}.
\end{example}

It is worth noticing that the group $\R^\N$ from item (i) of this example is not locally minimal; see \cite[Example 7.44]{DM}.
We now exhibit a minimal group distinguishing between properties \Zcm\ and \Zm.

\begin{example}
\label{list:cannot:be:extended}
Let $G$ be any countably infinite minimal metric abelian group, for example, $\Q/\Z$; see \cite{Do,S}. Then {\em $G$ is a (precompact) minimal abelian group satisfying \Zcm\ but failing \Zm\/}. Indeed, $G$ is precompact by Fact \ref{precompactness:theorem:for:minimal:groups}. Since $G$ is infinite, it is non-discrete.
The rest follows from Corollary \ref{groups:of:size:smaller:than:c}.
\end{example}

Our next  proposition shows that the previous two examples cannot be somehow ``combined together'' to obtain a (locally) precompact (sequentially) complete group distinguishing between properties \Zcm\ and \Zm.

\begin{proposition}
\label{locally:precompact:sequentially:complete}
Conditions \Zm\ and \Zcm\ are equivalent for locally precompact, sequentially complete groups. In particular, these two conditions 
coincide for locally countably compact groups.  
\end{proposition}

\begin{proof}
Let $G$ be a locally precompact, sequentially complete group. The implication \Zm$\to$\Zcm\ is established in Proposition \ref{arrows:between:Z}. To prove the reverse implication suppose that $G$ does not satisfy \Zm. Then $G$ has  a non-discrete closed zero-dimensional metric subgroup $N$. Since $N$ is a closed subgroup of $G$, it is locally precompact and sequentially complete. Since $N$ is also metrizable, $N$ is complete. Being also locally precompact, $N$ is locally compact.
Being  a non-discrete locally compact zero-dimensional group, $N$ contains an infinite open compact subgroup $C$, by van Dantzig's theorem \cite{vDan}. This shows that $G$ does not satisfy \Zcm. 
\end{proof}

Example \ref{example:Z} shows that local precompactness in  Proposition \ref{locally:precompact:sequentially:complete} is essential, while  Example \ref{list:cannot:be:extended} shows that sequential completeness is essential  as well.

Examples in this section clearly show that none of the implications  in \eqref{Z:implications} can be reversed, even under some ``mild'' compactness-like conditions. Additional examples in this direction having much stronger compactness-like properties can be also found in Section \ref{compact-like:examples}.
In this paper we study in detail the question of how one can strengthen these 
compactness-like conditions in order to make the implications  in \eqref{Z:implications} reversible. 

\section{Main results}

Recall that a topological group $G$ is called {\em almost  connected} \ if $G$ has a non-empty open connected subset \cite{DS_memo}. Clearly, a topological group $G$ is almost connected precisely when its connected component $c(G)$ is open.
Obviously, each connected group is almost connected.
All Lie groups are locally connected and locally connected groups are almost connected, so almost connectedness is a necessary condition for being a Lie group.

Lie groups are also locally compact. Combining these facts with (a part of) Diagram 1, we obtain a simplified diagram containing the most essential properties 
that appear in our main results. 

\medskip
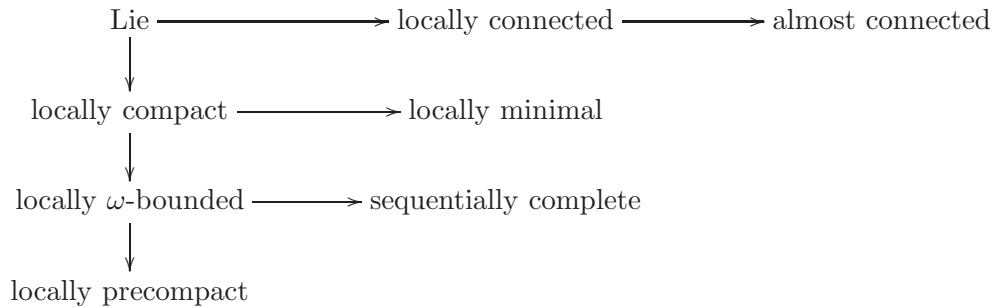
\begin{center}
${\xymatrix@!0@C5.0cm@R=1.2cm{
\mbox{Lie} \ar@{->}[r] \ar@{->}[d] & \mbox{locally connected} \ar@{->}[r]& \mbox{almost connected} \\
\mbox{locally compact} \ar@{->}[r] \ar@{->}[d] & \mbox{locally minimal} & \\
\mbox{locally $\omega$-bounded}\ar@{->}[d] \ar@{->}[r] & \mbox{sequentially complete} & \\
\mbox{locally precompact} & & \\
   }}$
\end{center}
\medskip
\begin{center}
Diagram 2.
\end{center}
\medskip

The Lie property is the ``maximal element'' of this diagram in a sense that it implies all other properties listed in it. The ``minimal elements'' of Diagram 2 (that is, the weakest properties that have no arrows starting at them) are {\em  almost connected, locally minimal,  sequentially complete\/} and  {\em locally precompact}.  The first property is clearly a weak connectedness-like property, while the other three are mild compactness-like properties. One might hope that these four general properties combined together and amended by conditions from Definition \ref{Z:def} restricting closed zero-dimensional subgroups of a given topological group would result in some sort of a characterization of Lie groups. In other words, one could hope that the Lie property  could be ``assembled'' from these properties taken as  ``building blocks''.  For example, it is quite tempting to suggest that Lie groups are precisely  the almost connected, locally minimal, locally precompact, 
sequentially complete groups having one of the three properties from Definition \ref{Z:def}. Unfortunately, this conjecture turns out to be false in a strong way, as the following couple of examples demonstrates.

\begin{example}
\label{two:examples}
(i) 
There exists a countably compact (thus, precompact and sequentially complete), connected, 
locally connected (so also almost connected) abelian group which has property \Zgen\ but is not Lie.

(ii) 
There exists a minimal, locally countably compact (so locally precompact and sequentially complete),  locally connected (thus, almost connected)
nilpotent group of class 2  which has property \Zgen\ but is not Lie.
\end{example}

These two examples  are built using quite sophisticated set-theoretic and group-theoretic machinery, with the help of the Continuum Hypothesis; the interested reader is referred to Examples 
\ref{Tkachenko:example} and \ref{loc_conn+min+loc_cc+Z-Lie}, respectively.
In Section \ref{compact-like:examples} one can also find additional examples (without additional set-theoretic assumptions beyond ZFC) with somewhat weaker compactness properties.

Inspired by these examples outlining the limits of what can be proved, we shall attempt to slightly strengthen the four building block conditions so that to obtain characterizations of (compact) Lie groups, sometimes restricting ourselves to the class of abelian groups. 

For example, by replacing ``locally precompact and sequentially complete'' with the stronger property of local  $\omega$-boundedness, one can obtain a characterization of Lie groups even without the assumption of local minimality or almost connectedness. Indeed, our first resut shows that the three properties from Definition \ref{Z:def} are equivalent for locally $\omega$-bounded groups.

\begin{theorem}
\label{thm:AB}
For a locally $\omega$-bounded group $G$ the following conditions are equivalent: 
\begin{itemize}
\item[(i)] $G$ is a Lie group; 
\item[(ii)] $G$ satisfies \Zgen; 
\item[(iii)] $G$ satisfies \Zm;
\item[(iv)] $G$ satisfies \Zcm.
\end{itemize}
\end{theorem}

The proof of this theorem is postponed until Section \ref{O-bou}.

Example \ref{two:examples}(i) shows that the implication (ii)$\to$(i) in this theorem need not hold if one replaces ``locally $\omega$-bounded'' by ``countably compact'', even in the abelian case. We refer the reader to Example \ref{cc:example:distinguishing:Zs}(ii) for an example of a countably compact abelian group satisfying \Zm\ without property \Zgen; therefore,  the implication (iii)$\to$(ii) need not hold if one replaces ``locally $\omega$-bounded'' by ``countably compact'', even in the abelian case. Nevertheless, items (iii) and (iv) remain equivalent even under this weaker assumption; see Proposition  \ref{locally:precompact:sequentially:complete}.

Since Lie groups are locally compact and locally compact groups are locally $\omega$-bounded, we get the following characterization of Lie groups in terms of their closed zero-dimensional compact metric subgroups.

\begin{corollary}
\label{characterizing:by:local:omega-boundedness}
A topological group is a Lie group if and only if it is locally $\omega$-bounded and has no infinite compact metric zero-dimensional subgroups.
\end{corollary}

Example \ref{two:examples}(i) shows that one cannot replace ``locally $\omega$-bounded'' by ``countably compact'' in this corollary, even in the abelian case.

Even the locally compact version of Corollary \ref{characterizing:by:local:omega-boundedness} seems to be new.

\begin{corollary}
\label{loc:compact:corollary}
A topological group is a Lie group if and only if it is locally compact and has no infinite compact metric zero-dimensional subgroups.
\end{corollary}

Our second theorem shows that,  in the class of abelian groups, the conjunction of local minimality and local precompactness suffices for characterizing Lie groups 
in the spirit of Theorem \ref{thm:AB}. Adding sequential completeness to the mix produces even more similarity with Theorem \ref{thm:AB}.

\begin{theorem}
\label{thm:NEW}
\label{thm:C}
For a locally minimal, locally precompact abelian group $G$  the following conditions are equivalent:   
\begin{itemize}
\item[(i)] $G$ is a Lie group; 
\item[(ii)] $G$ satisfies \Zgen; 
\item[(iii)] $G$ satisfies \Zm.
\end{itemize}
Moreover, if $G$ is additionally assumed to be sequentially complete, then the following condition can be added to this list:
\begin{itemize}
\item[(iv)] $G$ satisfies \Zcm.
\end{itemize}
\end{theorem}

The proof of this theorem is postponed until Section \ref{Min}.

It follows from Corollary \ref{Bohr:proposition} that local minimality of $G$ cannot be omitted in Theorem \ref{thm:C}.  Example \ref{list:cannot:be:extended} shows that one cannot drop the additional assumption of sequential completeness in the final part of Theorem \ref{thm:NEW}. As was mentioned in the beginning of this section, commutativity of the group cannot be dropped in Theorem \ref{thm:NEW}.

Since Lie groups are locally compact (so, locally minimal) and satisfy property \Zm, Theorem \ref{thm:NEW}  gives the following characterization of abelian Lie groups.

\begin{corollary}\label{new:corollary1}
An abelian topological group group $G$ is a Lie group if and only if $G$ is  locally minimal, locally precompact and all closed metric zero-dimensional subgroups of $G$ are discrete.
\end{corollary}

Example \ref{Hilbert:space} shows that local precompactness of $G$ cannot be omitted from  the assumptions of both Theorem \ref{thm:NEW} and Corollary \ref{new:corollary1}.

\begin{corollary}
\label{thm:NEW*}
For a minimal abelian group $G$  the following conditions are equivalent:   
\begin{itemize}
\item[(i)] $G$ is a compact Lie group; 
\item[(ii)] $G$ satisfies \Zgen; 
\item[(iii)] $G$ satisfies \Zm.
\end{itemize}
\end{corollary}
\begin{proof}
Since minimal abelian groups are precompact by Fact \ref{precompactness:theorem:for:minimal:groups}, the conclusion of our corollary follows from Theorem \ref{thm:NEW} and the fact that precompact Lie groups are compact.
\end{proof}

Example \ref{list:cannot:be:extended} shows that one cannot add \Zcm\ to the list of equivalent conditions in Corollary \ref{thm:NEW*}.

Since compact groups are minimal and precompact discrete groups are finite,  Corollary \ref{thm:NEW*} yields a concise
characterization of compact abelian Lie groups:

\begin{corollary}
\label{min:implies:compact:Lie}
An abelian topological group is a compact Lie group if and only if it is minimal and has no infinite closed metric zero-dimensional subgroups.
\end{corollary}

Example \ref{Hilbert:space} shows that minimality cannot be replaced with local minimality in Corollaries \ref{thm:NEW*} and \ref{min:implies:compact:Lie}.

Since the Lie groups are locally connected, it is natural to restrict the study of Lie groups within the class of locally connected groups. Actually one can work even under the 
weaker condition of almost connectedness. For almost connected groups, our third theorem extends Theorem \ref{thm:NEW} far beyond the abelian case.

\begin{theorem}
\label{last:theorem}
For  an almost  connected, locally  minimal, precompact sequentially complete group $G$, the following conditions are equivalent:
\begin{itemize}
\item[(i)] $G$ is a compact Lie group; 
\item[(ii)] $G$ satisfies \Zgen; 
\item[(iii)] $G$ satisfies \Zm;
\item[(iv)] $G$ satisfies \Zcm.
\end{itemize}  
\end{theorem}

The proof of this theorem is postponed until Section \ref{Min}.

Example \ref{Tkachenko:example} below shows that local minimality cannot be omitted in Corollary  \ref{new:corollary1} and Theorem \ref{last:theorem}.
On the other hand, Example \ref{MMDD} shows that a locally countably compact (so locally precompact and sequentially complete) 
locally connected minimal nilpotent group which satisfies \Zgen \ need not be a Lie group. This shows that ``precompact'' cannot be replaced 
by ``locally precompact'' in Theorem \ref{last:theorem}. 

\begin{corollary}\label{corollary:last:theorem} 
A topological group is a compact Lie group if and only if it is almost connected, sequentially complete, precompact, locally  minimal and all its compact metric zero-dimensional subgroups are finite. 
 \end{corollary}

Example \ref{Hilbert:space} shows that precompactness of $G$ is necessary in Theorem \ref{last:theorem} and cannot be dropped from its Corollary 
\ref{corollary:last:theorem}.

The  particular version of our results deserves special attention.

\begin{corollary}
\label{cc:minimal:is:Lie}
Let $G$ be a countably compact minimal group satisfying \Zcm. If $G$ is either abelian or almost connected, then $G$ is a compact Lie group.
\end{corollary}

\begin{proof} Indeed, $G$ is precompact and sequentially complete. Now the conclusion follows from Theorem \ref{thm:NEW} (in the abelian case) and 
Theorem \ref{last:theorem} (in the almost connected case).
\end{proof}

\begin{remark}
Since locally $\omega$-bounded groups are locally precompact and sequentially complete, the equivalence of items (iii) and (iv) in Theorem \ref{thm:AB} follows from Proposition \ref{locally:precompact:sequentially:complete}.
Similarly, this proposition implies also the additional statement in Theorem \ref{thm:NEW}, as well as the equivalence of items (iii) and (iv) in Theorem \ref{last:theorem}.
\end{remark}

One can consider the weaker versions of the three conditions \Zgen, \Zm\ and \Zcm\ from  Definition \ref{Z:def} obtained by replacing the word ``subgroup'' with the word ``normal subgroup''. The following example shows that (with the trivial exception of the purely ``abelian'' results) most of our results spectacularly fail for these weaker versions of the three properties.

\begin{example} 
Let $L=SO_3(\R)$ be a compact connected simple Lie group. Then {\em $G=L^\N$ is a compact connected metric group 
without non-trivial closed zero-dimensional normal subgroups, yet $G$ is not a Lie group\/}. Indeed, by a well-known theorem of Hofmann \cite{HM}, a closed zero-dimensional normal subgroup of a connected compact group must be central, and the conclusion follows from the fact that $G$ has the trivial center. 
\end{example}

It is worth mentioning here the TAP property from \cite{Sp} defined by requiring that no sequence in a topological group is multiplier convergent;
see \cite{DSS}. This property is weaker than NSS \cite{Sp}, and therefore, is possessed by every Lie group. Since TAP groups satisfy \Zcm, the results in this section can be applied to obtain characterizations of (compact) Lie groups in terms of multiplier convergence of sequences; see \cite{DSS1}.

We conclude this section with the diagram summarizing the main results exposed above (i.e., those from the present article, contained in the 
implications on the {\em first row}s of the diagram, and those from \cite{DP2,DSt}, contained in the  implications in the {\em lower part} of the diagram).

\bigskip
\begin{center}\ \ \ \ \ \ \ \ \ \ \ \ \ \ 
${\xymatrix@!0@C2.5cm@R=2.3cm{
  \mbox{Lie}   \ar@{=>}[r]|{(2)} 
  \ar@{->}[d]|{(5)} &  \mbox{\Zgen}  \ar@{=>}[r]|{(3)}&
  \mbox{\Zm}
  \ar@{=>}[r]|{(4)} &  \mbox{\Zcm}\ar@/_1.5pc/[lll]|{(1)} \ar@{->}[d]|{(6)} 
 \ar@{=>}[dr] & \\
\mbox{ \Ptd } \ar@{=>}[r]&   \mbox{ \Pt }  \ar@{=>}[r] &
\mbox{\Pmin}   
   \ar@{=>}[r]  & \mbox{\Avoid{\mathscr{P}^*}}
     \ar@/^1.5pc/[lll]|{(7)}   \ar@{=>}[r] & \mbox{\Avoid{\mathscr{P}}} \ar@/^1.2pc/[l]|{(8)} & \\
   }}$
\end{center}

\bigskip
\begin{center}
Diagram 3.
\end{center}
\bigskip

The double arrows denote implications that 
always hold.
The single arrows denote implications valid only in special classes of topological groups.

The implication (1) holds for the three classes of groups described in the assumptions of Theorems \ref{thm:AB}, \ref{thm:NEW} and \ref{last:theorem}, and the implications (2), (3) and (4) become equivalences for the same classes of groups.  However, these two implications are not equivalences in general, as  Examples \ref{example:Z}, \ref{list:cannot:be:extended}, \ref{Tkachenko:example} and \ref{cc:example:distinguishing:Zs} show. 

According to Theorem \ref{DSto}, the implication (7) holds  for locally compact groups (i.e., for locally compact groups  \Ptd, \Pt, \Pmin\ and \Avoid{\mathscr{P}^*}\ are equivalent).
The implications (5), (6)  and (8) hold  for compact groups, as in this case \Avoid{\{\Z\}} is vacuous, so \Ptd, \Pt, \Pmin,  \Avoid{\mathscr{P}}\ 
and \Avoid{\mathscr{P}^*}\ are equivalent in this case.

\section{Non-closed subgroups of abelian Lie groups do not have property \Zm}

The results in this section are used only in Section \ref{Min}, so the reader may skip it at first reading.

The next theorem is of independent interest extending way beyond the scope of our paper. 

\begin{theorem}
\label{R^n}
Let $k\ge 2$ be an integer. For every $i=1,\dots,k$ let $K_i$ be either $\R$ or $\T$. Let $G$ be a dense subgroup of $K=\prod_{i=1}^k K_i$ such that $G\cap H_i$ is not dense in $H_i$ for every $i=1,\dots,k$, where $H_i=\{(x_1,\dots,x_k)\in K: x_i=0\}$. Then $G$ is zero-dimensional.
\end{theorem}

\begin{proof}
Let $i=1,\dots,k$.  If $K_i=\R$, then 
we denote by $d_i$ the usual metric on $K_i$. In case $K_i=\T$, we define $d_i(x,y)$ as the length of the shortest arc connecting $x,y\in K_i$.
Clearly, $d_i$ is a translation-invariant metric on $K_i$. For $x\in K_i$ and $\varepsilon>0$, let 
$U_i(x,\varepsilon)=\{y\in K_i: d_i(x,y)<\varepsilon\}$ denote the $\varepsilon$-neighbourhood of $x$. 
For $x,y\in K_i$ with $d_i(x,t)<\pi$, we use $(x,y)$ and $[x,y]$ to denote the shortest open arc and the shortest closed arc in $K_i$ connecting $x$ and $y$.

Using our assumption, we can fix a real number $\delta>0$ and  a point $x_i=(x_{i}^{1},\dots,x_{i}^{k})\in H_i$ for every $i=1,\dots,k$ such that 
\begin{equation}
\label{open:sets:avoiding:G}
G\cap \left(\prod_{j=1}^{i-1}
U_j(x_{i}^{j},\delta)
\times 
\{0\}\times
\prod_{j=i+1}^k 
U_j(x_{i}^{j},\delta)
\right)
=\emptyset
\mbox{ for all }
i=1,\dots,k.
\end{equation}

Since $G$ is a topological group, in order to show that $G$ is zero-dimensional, it suffices to prove that $G$ has a a local base at its identity element 
$(0,\dots, 0)$  consisting of clopen subsets of $G$.  Fix $\varepsilon>0$ such that $\varepsilon<\min\{\delta/4, \pi/2\}$.

Fix $i=1,\dots,k$. Since $\varepsilon<\pi/2$, there are exactly two elements $a_i,b_i\in K_i$ such that $d_i(0,a_i)=d_i(0,b_i)=\varepsilon$.  Then 
\begin{equation}
\label{eq:V_i}
V_i=
\prod_{j=1}^{i-1}
U_j(-x_{i}^{j},\varepsilon)
\times 
(0,a_i)\times
\prod_{j=i+1}^k 
U_j(-x_{i}^{j},\varepsilon)
\end{equation}
and
\begin{equation}
\label{eq:W_i}
W_i=
\prod_{j=1}^{i-1}
U_j(-x_{i}^{j},\varepsilon)
\times 
(0,b_i)\times
\prod_{j=i+1}^k 
U_j(-x_{i}^{j},\varepsilon)
\end{equation}
are non-empty open subsets of $K$. Since $G$ is dense in $K$, we can choose 
\begin{equation}
\label{g_i:h_i}
g_i=(g_i^1,\dots,g_i^k)\in G\cap V_i
\ \ 
\mbox{ and }
\ \ 
h_i=(h_i^1,\dots,h_i^k)\in G\cap W_i.
\end{equation}
Define $y_i=x_i+g_i$ and $z_i=y_i+h_i$. Let
\begin{equation}
\label{def:y_i:and:z_i}
y_i=(y_i^1,\dots,y_i^k)
\ \ 
\mbox{ and }
\ \ 
z_i=(z_i^1,\dots,z_i^k).
\end{equation}
Note that  $g_i^i\in (0,a_i)$  by \eqref{eq:V_i} and \eqref{g_i:h_i}.
Similarly, $h_i^i\in (0,b_i)$ by \eqref{eq:W_i} and \eqref{g_i:h_i}.
Since $x_i^i=0$, it follows from  $y_i=x_i+g_i$, $z_i=y_i+h_i$
and \eqref{def:y_i:and:z_i} that $y_i^i=g_i^i\in (0,a_i)$ and $z_i^i=h_i^i\in (0,b_i)$. Therefore, 
\begin{equation}
\label{interval:y_i^i:z_i^i}
0\in (y_i^i,z_i^i)\subseteq (0,a_i)\cup(0,b_i)=U_i(0,\varepsilon)
\end{equation}
by our choice of $a_i$ and $b_i$. It follows that
$$
O=\prod_{i=1}^k (y_i^i,z_i^i)
$$
is an open subset of $K$ with
\begin{equation}
\label{basic:cube}
(0,\dots,0)\in O\subseteq \prod_{j=1}^k U_j(0,\varepsilon).
\end{equation}
Clearly, the boundary $\mathrm{Bd}(O)$ of $O$ in $K$ has the form 
\begin{equation}
\label{boundary:form}
\mathrm{Bd}(O)=\bigcup_{i=1}^k 
\left(\prod_{j=1}^{i-1}[y_{i}^{j},z_{i}^{j}]
\times 
\{y_{i}^{i},z_{i}^{i}\}\times
\prod_{j=i+1}^k [y_{i}^{j},z_{i}^{j}]\right).
\end{equation}

\begin{claim}
\label{boundary:claim}
If $c=(c^1,\dots,c^k)\in \mathrm{Bd}(O)$, then  there exists $i=1,\dots,k$ such that
$c^i\in \{y_{i}^{i},z_{i}^{i}\}$ and
$c^j\in 
U_j(y_i^j, \delta)\cap U_j(z_i^j,\delta)$
for $j=1,\dots,k$
with 
$j\not=i$.
\end{claim}
\begin{proof}
Use \eqref{boundary:form}
to fix $i=1,\dots,k$ with
\begin{equation}
\label{a:in:a:face}
(c^1,\dots,c^k)\in 
\prod_{j=1}^{i-1}[y_{i}^{j},z_{i}^{j}]
\times 
\{y_{i}^{i},z_{i}^{i}\}\times
\prod_{j=i+1}^k [y_{i}^{j},z_{i}^{j}].
\end{equation}
In particular, 
$c^i\in \{y_{i}^{i},z_{i}^{i}\}$.
Let $j=1,\dots,k$ and $j\not=i$.
Since $g_i^j, h_i^j\in U_j(-x_{i}^{j},\varepsilon)$ 
by \eqref{eq:V_i}, \eqref{eq:W_i} and \eqref{g_i:h_i},
from $y_i=x_i+g_i$, $z_i=y_i+h_i$
and \eqref{def:y_i:and:z_i}
we conclude that
$y_i^j,z_i^j\in U_j(0,\varepsilon)$, 
which yields also
$c^j\in [y_i^j,z_i^j]\subseteq U_j(0,\varepsilon)$ by \eqref{a:in:a:face}.
Since  $\varepsilon<\delta/4$, it follows that
$U_j(0,\varepsilon)\subseteq U_j(y_i^j,\delta)\cap U_j(z_i^j,\delta)$. 
Therefore, 
$c^j\in U_j(y_i^j,\delta)\cap U_j(z_i^j,\delta)$. 
\end{proof}

\begin{claim}
\label{boundary:is:disjoint:from:G}
$\mathrm{Bd}(O)\cap G=\emptyset$.
\end{claim}
\begin{proof}
Let $c=(c^1,\dots,c^k)\in \mathrm{Bd}(O)$. We are going to show that $c\not\in G$. Let $i$ be as in the conclusion of Claim \ref{boundary:claim}. We consider two cases.

\smallskip
{\em Case 1\/}. $c^i=y_{i}^{i}$.
In this case $c^i-g_i^i=y_{i}^{i}-g_i^i=x_i^i$. Let $j=1,\dots,k$ and $j\not=i$.
Note that $c^j\in U_j(y_i^j, \delta)$ by Claim \ref{boundary:claim}.
Since the metric $d_j$ is translation invariant,
$$
c^j-g_i^j\in U_j(y_i^j, \delta)-g_i^j=U_j(y_i^j-g_i^j, \delta)
=
U_j(x_i^j, \delta).
$$
This shows that 
$$
c-g_i\in \prod_{j=1}^{i-1}
U_j(x_{i}^{j},\delta)
\times 
\{x_{i}^{i}\}\times
\prod_{j=i+1}^k 
U_j(x_{i}^{j},\delta).
$$
Now \eqref{open:sets:avoiding:G} gives $c-g_i\not\in G$. Since $g_i\in G$ and $G$ is a subgroup of $K$, it follows that $c\not\in G$.

\smallskip
{\em Case 2\/}. $c^i=z_{i}^{i}$.
In this case $c^i-h_i^i=z_{i}^{i}-h_i^i=x_i^i$. Let $j=1,\dots,k$ and $j\not=i$.
Note that $c^j\in U_j(z_i^j, \delta)$ by Claim \ref{boundary:claim}. Since  the metric $d_j$ is translation invariant,
$$
c^j-h_i^j\in U_j(z_i^j, \delta)-h_i^j=U_j(z_i^j-h_i^j, \delta)
=
U_j(x_i^j, \delta).
$$
This shows that 
$$
c-h_i\in \prod_{j=1}^{i-1}
U_j(x_{i}^{j},\delta)
\times 
\{x_{i}^{i}\}\times
\prod_{j=i+1}^k 
U_j(x_{i}^{j},\delta).
$$
Now \eqref{open:sets:avoiding:G} gives $c-h_i\not\in G$. Since
$h_i\in G$ and $G$ is a subgroup of $K$, it follows that  $c\not\in G$.
\end{proof}

Since $G$ is dense in $K$, from Claim \ref{boundary:is:disjoint:from:G} one easily concludes that $O\cap G$ is a clopen subset of $G$.
For each small enough $\varepsilon>0$ we have found an open subset $O$ of $K$ satisfying \eqref{basic:cube}
such that $O\cap G$ is a clopen subset of $G$. This proves that $G$ has a clopen base at $(0,\dots, 0)$.
\end{proof}

\begin{corollary}\label{corollary:R^n}
Let $G$ be a dense subgroup of $K=\R^m\times\T^n$, where $m, n\in \N$. If $G$ satisfies \Zm, then $G=K$. 
\end{corollary}

\begin{proof}
We prove our corollary by induction on $k=m+n$. For $k=1$ the assertion follows from the obvious fact that all proper dense subgroups of $\R$ or $\T$ are zero-dimensional,  so do not satisfy \Zm. Suppose that $k\in\N$, $k\ge 2$ and our 
corollary has already been proved $k-1$. For $i=1,\dots,k$, define $K_i=\R$ for $i\le m$ and $K_i=\T$ for $m<i<k$, so that $K=\prod_{i=1}^k K_i$.
Since $G$ is dense in $K$, it is non-discrete. Since $G$ is satisfies \Zm, it cannot be zero-dimensional. 
Applying Theorem \ref{R^n}, we conclude that there exists $i=1,\dots,k$ such that $G_i=H_i \cap G$ must be dense in $H_i=\{(x_1,\dots,x_k)\in K: x_i=0\}$.
Since $G_i$ is a closed subgroup of the group $G$ having property \Zm,
$G_i$ also has the same property. Since $G_i$ is dense in $H_i$, we can apply our inductive assumption to the pair of $G_i$ and $H_i$ to conclude that
$G_i = H_i$. Now, the splitting $K = H_i \times K_i$ produces the splitting
$G = H_i \times (G\cap K_i)$, as $H_i \times \{0\}\subseteq G$. The density of $G$ in $K$ yields that $G\cap K_i$ is dense in $K_i$. Since $G\cap K_i$
is a closed subgroup of $G$, it must satisfy \Zm.
Therefore, $G\cap K_i = K_i$ by the basis of our induction, and so $\{0\}\times K_i\subseteq G$.
Since $H_i\times\{0\}\subseteq G$ holds as well, we get  $K=H_i\times K_i\subseteq G$.
\end{proof}

\begin{corollary}\label{subgroups:of:Lie:having:Zm}
Let $G$ be a subgroup of an abelian Lie group $K$. If $G$ satisfies \Zm, then $G$ is closed in $K$; in particular $G$ is a Lie group itself.
\end{corollary}

\begin{proof} Since closed subgroups of Lie groups are Lie, we may assume, without loss of generality, that $G$ is dense in $K$.  It is known that $K \cong \R^m\times \T^n \times D$ for some $m, n\in \N$ and some discrete abelian group $D$ \cite{HM}. Now  $c(K)=\R^m\times \T^n \times \{0\}$ is a clopen subgroup of $K$, so $G^*=G\cap c(K)$ is a dense subgroup of $c(K)\cong \R^m\times \T^n$ satisfying \Zm, being a closed subgroup of $G$.  Therefore, $G^*= c(K)$ by Corollary \ref{corollary:R^n}; that is, $\R^m\times \T^n \times \{0\}\subseteq G$. Since $G$ is dense in $K$ and $D$ is discrete, one easily concludes that $\{0\}\times\{0\}\times D\subseteq G$. This gives $K=\R^m\times \T^n \times D\subseteq G$.
\end{proof}

\begin{remark}
The Lie group $\R$ contains the group $\mathbb{Q}$ of rational numbers as its dense proper (thus, non-closed) subgroup. Since
$\mathbb{Q}$ is countable,  it has property \Zcm\ by Corollary \ref{small:groups:are:Zcm}. This example shows that both
Corollaries \ref{corollary:R^n} and \ref{subgroups:of:Lie:having:Zm} fail if one replaces the assumption ``$G$ has property \Zm'' in them with the weaker assumption   ``$G$ has property \Zcm''.
\end{remark}

\section{The structure of almost connected compact groups with Lie center}

In this section  we obtain a structure theorem for almost connected compact groups having Lie center; see 
Theorem \ref{almost:connected}.
In order to achieve this, we use the dynamical 
properties of the action of the group on the semisimple derived group of its connected component.

We shall need the following folklore fact concerning the dichotomy related to normal subgroups of a direct product containing a simple non-abelian direct summand.  For the sake of completeness, we include  its  proof.
 
 \begin{lemma}
\label{dichotomy:for:product:of:two}
Let  $G = H\times S$ be the direct product  of a group $H$ and a simple non-abelian group $S$,
and let $p_S: G\to S$ be the projection on the second coordinate. If $N$ is a normal subgroup of $G$ and
$p_S(N)\not=\{e_S\}$, then $\{e_H\} \times S\subseteq N$. In particular, either $N \subseteq H \times \{e_S\}$ or $\{e_H\} \times S\subseteq N$.
\end{lemma}

\begin{proof} 
Since $p_S(N)\not=\{e_S\}$, we can pick
an element $n= (h,s) \in N$ with $s\ne e_S$. Since $Z(S)$ is trivial, there exists  $b\in S$ such that $[b,s]= bsb^{-1}s^{-1} \ne e_S$. Let $x = (e_H,b) \in G$. Then $[x,n]= xnx^{-1}n^{-1} = (e_H, [b,s]) \in  \{e_H\} \times S$. As $N$ is a normal subgroup of $G$, one has $[x,n] \in N$. Therefore,
$e_S\ne [b,s]=p_S([x,n])\in p_S(N)$.
This proves that $p_S(N)$ is a non-trivial normal subgroup of $S$. Since $S$ is simple, we conclude that $p_S(N) = S$. The above argument shows that if $s\in p_S(N)$, then $(e_H, [b,s])  \in (\{e_H\} \times S) \cap N$. As $p_S(N) = S$, this proves that $ \{e_H\} \times S' \subseteq (\{e_H\} \times S) \cap N$. Since $S$ is simple, $S' = S$, hence  $\{e_H\} \times S \subseteq N$. 
\end{proof} 

\begin{lemma}
\label{simple:subgroups:coincides:with:the:factor}
Assume that $\{L_i: i\in I\}$ is a family of simple non-abelian topological groups, 
$R=\prod_{i\in I} L_i$ is its direct product equipped with the Tychonoff product topology, $M$ is a topological group 
and 
$N$ is a simple closed normal subgroup of the product $M\times R$.
Then either $N\subseteq M$ or $N=L_k$ for some $k\in I$.
\end{lemma}
\begin{proof}
Suppose that $N$ is not a subgroup of $M$. Then
there exists $k\in I$ such that $p_k(N)\not=\{e_{L_k}\}$, where $p_k: M\times R\to L_k$ is the projection.
Applying Lemma \ref{dichotomy:for:product:of:two} to $H=M\times \prod_{i\in I\setminus \{k\}}$ and $S=L_k$, we conclude that $L_k\subseteq N$. (Here we identify $L_k$ with the subgroup of $M\times R$ in a natural way.) Since $L_k$ is a closed normal subgroup of $M\times R$ (and thus, also of $N$), from the simplicity of $N$ one gets $L_k=N$.
\end{proof}

Given an action $G\times X\to X$, $(g,x)\mapsto gx$ for $g\in G$ and $x\in X$, of a group $G$ on a set $X$, we say that a subset $Y$ of $X$ is {\em $G$-invariant\/} if
$gy\in Y$ whenever $g\in G$ and $y\in Y$.

\begin{lemma}
\label{dynamical}
Suppose that $I$ is an infinite set, $\{L_i:i\in I\}$ is a family of 
simple non-abelian
topological groups, $R= \prod_{i\in I} L_i$ is their Cartesian product equipped with the Tychonoff product topology, 
and
$G\times R\to R$ is an action of a group $G$ on $R$ by  continuous 
automorphisms
such that
$\{g(L_i):g\in G\}$ is finite 
for every 
$i\in I$. 
Then there exists a partition $I=\bigcup\{I_j:j\in J\}$ of $I$ into non-empty finite sets $I_j$ such that  $N_j=\prod_{i\in I_j} L_i$ is $G$-invariant for every $j\in J$.
\end{lemma}

\begin{proof} Let $g\in G$ and $j\in I$ be arbitrary. 
Since $g$ is a continuous automorphism of $R$,  $g(L_j)$ is a closed normal subgroup of $R$.
Since $g(L_j) \cong L_j$ is simple, 
Lemma \ref{simple:subgroups:coincides:with:the:factor} applied to $M=\{0\}$ and $N=g(L_j)$ allows us to find
a unique index $\phi(g,j) \in I$ such that $g(L_j) = L_{\phi(g,j)}$.
This argument shows that the action of $G$ on $R$ induces an action $\phi: G \times I \to I$ of $G$ on the index set $I$.
By our assumption on the action $G\times R\to R$, the orbit 
$I_j=\{\phi(g,j):g\in G\}$ of each $j\in I$ is finite.
Choose $J\subseteq I$ such that $\{I_j: j\in J\}$ is a faithful enumeration of all orbits of $I$. Then $I=\bigcup\{I_j:j\in J\}$ is the desired decomposition,
as $j\in I_j\not=\emptyset$ 
and $N_j=\prod_{i\in I_j} L_i$ is clearly $G$-invariant for every $j\in I$.
\end{proof}

We shall need in the sequel the following theorem collecting  useful facts about the structure of compact connected groups.
Items (a)--(c) are well known; see \cite[Theorem 9.24]{HM}. Item (b) is known as theorem of Varopoulos \cite{V}.
To the best of our knowledge, item (d) is new, and it is this item that will be essentially used in the proof of Theorem \ref{almost:connected} and Lemma \ref{compact:connected:tds:is:Lie} below.

\begin{theorem}\label{compact:connected:group:with:Lie:center}
Let $K$ be a compact connected group. 
\begin{itemize}
  \item[(a)] $K'$ is a connected compact group, $K = K'\cdot Z(K)$ and $Z(K')= Z(K) \cap K'$ is zero-dimensional.
  \item[(b)] $K/Z(K)=\prod_{i\in I}L_i$ is a product of non-trivial  simple compact connected Lie groups $L_i$.
  \item[(c)] For $i\in I$, let $\widetilde L_i$ be the covering group of $L_i$ and $Z_i = Z(\widetilde L_i)$. Then each $Z_i$ is finite and  $\widetilde L_i/Z_i\cong L_i$. Define $ L = \prod_{i\in I}\widetilde L_i $  and $Z = Z(L) =\prod_{i\in I} Z_i$.  Then there exists a closed subgroup $N$ of $Z$ such that $L/N \cong K'$ and $Z/N \cong Z(K')$. 
\item[(d)] If $Z(K)$ is a Lie group, then the following are equivalent: 
\begin{itemize}
  \item[(d$_1$)] $K$ is not a Lie group; 
  \item[(d$_2$)] $I$ is infinite; 
  \item[(d$_3$)] $I$ is infinite and there exist a finite $I_0 \subseteq I$, a 
closed characteristic Lie subgroup $M$ of $K'$  and a (closed)
 characteristic subgroup $R \cong \prod_{i\in I\setminus I_0}L_i$ of $K'$ such that $K' \cong M \times R$; in particular,  all normal subgroups of $M$ and $R$ are also normal subgroups of $K$. 
\end{itemize}
\end{itemize}
 \end{theorem}

\begin{proof} We provide only the proof of item (d). We assume from now on that $Z(K)$ is a Lie group. The equivalence (d$_1$) $\leftrightarrow$ (d$_2$) follows from (b) and the obvious fact that $\prod_{i\in I}L_i$ is a Lie group precisely when $I$ is finite.  The implication (d$_3$) $\to$ (d$_2$) is obvious, so we are left with the proof of the remaining  implication (d$_2$) $\rightarrow$ (d$_3$). Hence, we assume from now on that the set $I$ is infinite.

As a closed subgroup of the compact Lie group $Z(K)$, $Z(K')$ is a compact Lie group as well. Since $Z(K')$ is zero-dimensional by (a),  it is finite. Since $Z/N \cong Z(K')$ by (c), the subgroup $N$ of the compact abelian group $Z$ has finite index in $Z$. Taking the duals and using the fact that  $\widehat Z_i$ is finite (as each $Z_i$ is finite by (c)), we deduce that the annihilator $N^\bot=\{\chi\in\widehat{Z}: \chi(N)\subseteq \{0\}\}$  of $N$ is a finite subgroup  of $\widehat Z \cong \bigoplus_{i\in I}\widehat   Z_i$. Hence there exists a finite subset $I_0$ of $I$ such that  $N^\bot \subseteq \bigoplus_ {i\in I_0} \widehat  Z_i$.  Taking the duals once again, we conclude that $N$  contains the subgroup 
\begin{equation}
\label{eq:N_1}
 N_1 := \prod_{i\in I\setminus I_0}Z_i =\left(\bigoplus_{i\in I_0} \widehat  Z_i\right)^\bot
\end{equation}
 of $Z$. 
Since $N\subseteq Z=\prod_{i\in I_0}Z_i \times \prod_{i\in I\setminus I_0}Z_i
=
\left(\prod_{i\in I_0}Z_i\right)\times N_1$
and $N_1\subseteq N$, there exists a subgroup $N_0$ of  $\prod_{i\in I_0}Z_i$
such that 
$N = N_0\times N_1$.

Let $J=\{i\in I_0: Z_i\subseteq N\}$ and $I_1=I_0\setminus J$.
Since $J$ is finite, 
$N_0=A\times \prod_{i\in J} Z_i$ for some subgroup $A$ of $\prod_{i\in I_1} Z_i$. Now 
$$
N=\left(A\times \prod_{i\in J} Z_i\right)\times N_1
=
A\times \prod_{i\in I_1\cup J} Z_i
=
A\times \prod_{i\in I\setminus I_1} Z_i.
$$
by \eqref{eq:N_1}.
By our construction, $A$ does not contain any subgroup $Z_i$ for $i\in I_0\setminus J=I_1$. This argument allows us to assume, without loss of generality,
that 
$N_0$ contains no $Z_i$ for $i\in I_0$.
This additional assumption means that 
\begin{equation}\label{Z:Z}
N_0 \cap Z_i \ne Z_i\  \mbox{ for all }\ i \in I_0.
\end{equation}

Define
\begin{equation}
\label{M:and:R:tildas}
\widetilde M = \prod_{i\in I_0}\widetilde L_i
\ \ 
\mbox { and }
\ \ 
\widetilde R=\prod_{i\in I\setminus I_0}\widetilde L_i,
\end{equation}
so that $L=\prod_{i\in I} \widetilde L_i=\widetilde M\times \widetilde R$.
Since 
$N_0$ is a subgroup of $\widetilde M$,
$N_1$ is a subgroup of $\widetilde R$
and 
$N=N_0\times N_1$,
it follows that 
$L/N=\widetilde M/N_0\times \widetilde R/N_1$.
Let $q: L/N \to K'$ be the isomorphism provided by (c).
Then
$K'=q(L/N)=q(\widetilde M/N_0\times \widetilde R/N_1)
=
q(\widetilde M/N_0)\times q(\widetilde R/N_1)
=
M\times R$, 
where
$$
M=q(\widetilde M/N_0)
\ \ 
\mbox{ and }
\ \ 
R=q(\widetilde R/N_1).
$$

Since $N_1= \prod_{i\in I\setminus I_0}Z_i$,
we get
$$
R
\cong
\widetilde R/N_1=\left(\prod_{i\in I\setminus I_0}\widetilde L_i\right)
/
\left(\prod_{i\in I\setminus I_0}Z_i\right)
=
\prod_{i\in I\setminus I_0}\widetilde L_i/Z_i
=
\prod_{i\in I\setminus I_0}L_i.
$$

Since $I_0$ is finite and all $\widetilde L_i$ are Lie groups,
from 
\eqref{M:and:R:tildas}
we deduce that
$\widetilde M$ is a Lie group. 
Since $N_0$ is a finite subgroup of the compact group $\widetilde M$, 
the quotient $\widetilde M/N_0$ and 
$\widetilde M$ are locally homeomorphic. Therefore,
$\widetilde M/N_0$ is a Lie group.
From 
$M\cong  \widetilde M/N_0$, we deduce that $M$ is a Lie group as well.

Observe that $Z(\widetilde M) = \prod_{i\in I_0}  Z_i$ contains $N_0$. Moreover, $Z(\widetilde M/N_0) = Z(\widetilde M)/N_0$, so 
$$
(\widetilde M/N_0)/Z(\widetilde M/N_0) \cong (\widetilde M/N_0)/(Z(\widetilde M)/N_0) \cong \widetilde M/Z(\widetilde M) \cong \prod_{i\in I_0}  L_i.
$$

From this and the fact that $q\restriction_{\widetilde M/N_0}: \widetilde M/N_0 \to M$ is an isomorphism, we get
\begin{equation}\label{M:ZM}
M/Z(M) \cong  \widetilde M/Z(\widetilde M) \cong \prod_{i\in I_0}  L_i.
\end{equation}
Let $\psi: M/Z(M)\to \prod_{i\in I_0}  L_i$ be the isomorphism
witnessing \eqref{M:ZM}.
Let $\xi: M \to M/Z(M)$ be the canonical homomorphism, and let 
$\varphi= \psi\circ\xi: M\to \prod_{i\in I_0}  L_i$.
Clearly, $\ker \varphi = \ker \xi = Z(M)$,
so $\varphi$ has finite kernel. From this it easily follows that 
\begin{equation}
\label{preservation:with:finite:images}
S\cong \varphi(S)
\mbox{ for every infinite simple closed subgroup }
S
\mbox{ of }
M.
\end{equation}

\begin{claim}
\label{no:special:subgroups}
$M$ contains no non-abelian simple connected closed normal subgroups.
\end{claim}
\begin{proof}
For every $i\in I_0$ denote by $M_i$ the image of the subgroup $\widetilde L_i$ under the composition of the canonical map $\widetilde M \to \widetilde M/N_0$
and the   isomorphism $q\restriction_{\widetilde M/N_0}: \widetilde M/N_0 \to M$. Then $M_i$ is a normal connected Lie subgroup of $M$
and 
$\varphi(M_i)=L_i$.

Let $S$ be a non-abelian simple connected closed normal subgroup of $M$.
Clearly, $S$ is infinite.
Since $\varphi$ is a continuous homomorphism and $M$ is compact,
$\varphi(S)$ is a non-trivial closed normal subgroup of  $ \prod_{i\in I_0} L_i$.
Since $S$ is infinite,
$\varphi(S) \cong S$
by \eqref{preservation:with:finite:images}.
Since $S$ is simple, so is $\varphi(S)$.
Applying Lemma \ref{simple:subgroups:coincides:with:the:factor}
to $M=\{0\}$, $I=I_0$ and $N=\varphi(S)$,
we conclude that
$\varphi(S) = L_k$
for some $k \in I_0$. Since also $\varphi(M_k) = L_k$, we deduce that 
$$
Z(M) S = \varphi^{-1}(\varphi(S))=\varphi^{-1}(L_k)= Z(M)  M_k.
$$
 Since $Z(M)$ is finite, both $S$ and $M_k$ are finite-index subgroups of the group $H = Z(M) S=  Z(M)  M_k$.
Since 
both $S$ and $M_k$ are connected,
$S=c(H)=M_k$.
Since $S$ is simple, it is center-free.
To get a contradiction, it suffices to show that 
$M_k$ has a non-trivial center.
Since $N_0\subseteq \prod_{i\in I_0} Z_i=\prod_{i\in I_0} Z(\widetilde L_i)$,
we have $N_0\cap \widetilde L_k=N_0\cap Z_k$.
Combining this with \eqref{Z:Z}, we get $N_0\cap \widetilde L_k\not=Z_k$.
Therefore,
$Z(M_k)$ contains a subgroup isomorphic to the non-trivial 
subgroup $Z_k/(N_0 \cap \widetilde L_k)$. 
\end{proof}

\begin{claim}
Both $M$ and $R$ are characteristic subgroups of $K'$. 
\end{claim}
\begin{proof}
To
prove that $M$ is characteristic in $K'$, 
suppose the contrary.
Then there exists a continuous automorphism $a$ of $K'$
such that $M\setminus a(M)\not=\emptyset$.
Then there exists $k\in I\setminus I_0$ such that $p_k(a(M))\not=\{e_{L_k}\}$.
Clearly,
$a(M)$ is a closed 
normal subgroup of $K' = M \times R$.
Applying Lemma \ref{dichotomy:for:product:of:two} to 
$H=M\times\prod_{i\in I\setminus (I_0\cup\{k\})} L_i$, $S=L_k$ and $N=a(M)$,
we get
$L_k\subseteq a(M)$, which in turn implies 
$a^{-1}(L_k)\subseteq M$.
Since $a^{-1}$ is a continuous automorphism of $K'$ 
and $L_k$ is a closed normal subgroup of $K'$,  
$a^{-1}(L_k)$ 
is a closed normal subgroup of $K'$.
Furthermore, $a^{-1}(L_k)\cong L_k$, so  
$a^{-1}(L_k)$ is a non-abelian simple connected group.
This contradicts Claim
\ref{no:special:subgroups}.
This shows that $a(M)\subseteq M$ for every continuous automorphism 
$a$ of $K'$; that is, $M$ is characteristic in $K'$.
 
Next, let us prove that $R$ is a 
characteristic subgroup of $K'$. Let  $a$ be an arbitrary continuous automorphism of $K'$. 
Fix $i\in I\setminus I_0$ arbitrarily. The same argument as above shows that $a(L_i)$ 
cannot be contained entirely in $M$.
By Lemma \ref{simple:subgroups:coincides:with:the:factor},
$a(L_i)=L_k$ for some $k\in I\setminus I_0$. 
In particular, $a(L_i)\subseteq R$. 
Since $i$ was chosen arbitrarily,
$a(\bigoplus_{i\in I\setminus I_0} L_i)\subseteq R$.
Since $R$ carries the Tychonoff product topology,
$\bigoplus_{i\in I\setminus I_0} L_i$ is dense in $R$.
Since $a$ is continuous,
it follows that
$a(R)=a(\overline{\bigoplus_{i\in I\setminus I_0} L_i})
\subseteq 
\overline{R}=R$, where the bar denotes the closure in $K'$.
\end{proof}

Since 
both $M$ and $R$ are characteristic subgroups of $K'$,
all normal subgroups of $M$ and $R$ are normal subgroups of $K'$. Since $K = K'\cdot Z(K)$ by (a), this yields that all normal subgroups of $M$ and $R$ are also normal subgroups of $K$. 
\end{proof}

The next theorem describes the structure of almost connected compact groups with Lie center.
\begin{theorem}\label{almost:connected}
Let $K$ be an  almost connected compact group
such that 
$Z(K)$ is a Lie group
yet
$K$ itself is not a Lie group.
Then 
$c(K)'$ is topologically isomorphic to an infinite product of non-trivial closed normal Lie subgroups of $K$.
\end{theorem}
 
\begin{proof} 
For the sake of brevity, we let $K_0=c(K)$. 
Since $K$ is almost connected, $K_0$ is an open subgroup if $K$. Since $K$ is compact, $K_0$ has finite index in $K$. In particular, 
$K_0$ is a Lie group if and only if $K$ is.  Since $K$ is not a Lie group by our assumption, we conclude that $K_0$ is also not a Lie group. 
Being a closed subgroup of the Lie group $Z(K)$,  $Z(K_0)$ is a Lie group. 

Applying  the implication (d$_1$)$\to$(d$_3$) of Theorem \ref{compact:connected:group:with:Lie:center}(d) to $K_0$, we
find an infinite set $I$,
a family $\{L_i: i\in I\}$ of 
non-trivial closed 
normal Lie subgroups of
$K_0$,
a closed 
characteristic Lie subgroup $M$ of $K_0'$ and a (closed)
characteristic subgroup $R \cong \prod_{i\in I}L_i$ of $K_0'$
such that $K_0' \cong M \times R$.
(To simplify notations, we renamed the set $I\setminus I_0$ in Theorem \ref{compact:connected:group:with:Lie:center}(d) to $I$.)
While all $L_i$ are normal subgroups of
$K_0'$, they need not in general be normal in the bigger group $K$.
In the remaining part of the proof we shall find a remedy to this problem
by making use of Lemma \ref{dynamical}.
To this end, we need to define an appropriate group action of some group $G$ 
on $R$ via continuous
automorphisms satisfying the hypotheses of that lemma. 

Since the subgroup $K_0$ of $K$ is normal, the group $K$ acts on $K_0$ via conjugation. Denote by $G$ the group of (continuous) automorphisms of $K_0$ obtained in this way. 
Since 
$K'_0$ is a fully invariant subgroup of $K$, 
it is also $G$-invariant.
This determines an action of $G$ on $K_0'$ by restriction. (Note that 
while the action of $K$ on $K_0$ is via {\em internal\/} automorphisms, the action of $G$ on $K_0'$ is via automorphisms that need not be internal.) 
Since $M$ and $R$ are characteristic subgroups of $K_0'$, they are both $G$-invariant. In particular, $G$ acts also on $R$ by restriction.

In order to apply Lemma \ref{dynamical}, we need to check that 
$\{g(L_i):g\in G\}$ is finite 
for every 
$i\in I$. 
Let $N$ be an arbitrary normal subgroup of $K_0'$ (for example, any $L_i$).
Since $K_0$ has finite index in $K$, there exists a finite set $E\subseteq K$ such that $K=\bigcup\{xK_0:x\in E\}$.
Since $K_0=K_0'\cdot Z(K_0)$
by Theorem \ref{compact:connected:group:with:Lie:center}(i), 
we get $\{g(N):g\in G\}\subseteq \{x^{-1}N x:x\in E\}$.
In particular, $|\{g(N):g\in G\}|\le |E|$. 

Let $I=\bigcup\{I_j:j\in J\}$ be the partition as in the conclusion of Lemma \ref{dynamical}. Fix $j\in J$.
Since  all $L_i$ are Lie groups, their finite product $N_j=\prod_{i\in I_j} L_i$ is a Lie group. Since $N_j$ is $G$-invariant, it is a normal subgroup of $K$ by our choice of $G$. 
Since $M$ is characteristic in $K_0'$, it is $G$-invariant, and so $M$ is also a normal subgroup of $K$.
Finally, note that 
$$
c(K)'=K_0'=\cong M\times R\cong M\times\prod_{i\in I} L_i\cong M\times \prod_{j\in J} \left(\prod_{i\in I_j} L_i\right)=M\times \prod_{j\in J} N_j
$$
is the desired decomposition.
\end{proof}

\section{Proof of Theorem \ref{thm:AB}}
\label{O-bou}

Our first lemma in this section is a particular case of a much stronger result from \cite{DP2}. However, since this weaker version has a proof that is much simpler than the proof from \cite{DP2}, and since the manuscript \cite{DP2} is not easily accessible, we provide here this simpler direct proof for the reader's convenience.

\begin{lemma}
\label{exotic:tori:lemma}
If $G$ is a compact abelian group satisfying 
\Zcm, then $G\cong F\times \T^n$, where $F\cong t(\widehat{G})$ is a finite group and $n=\dim G$ is a non-negative integer number; in particular, $G$ is a Lie group.
\end{lemma}

\begin{proof} 
Let $X$ denote the discrete Pontryagin dual $\widehat{G}$ of $G$. Let $M$ be a maximal independent  subset of $X$, so that the subgroup $F$ generated by $M$ is free and $X/F$ is torsion. Let $N = F^\perp$ be the annihilator of $F$ in $G$. Then $N$ is a closed subgroup of $G$ with $N\cong \widehat{X/F} $. Since $X/F$ is torsion, the group $N\cong \widehat{X/F} $ is zero-dimensional. By our hypothesis, $N$ must be finite. Therefore, $F$ is a finite-index subgroup of $X$, i.e., $X/F$ is finite.  Let $h: X \to X/F$ be the canonical projection. Since the torsion subgroup $t(X)$ of $X$  trivially intersects $F$, $h\restriction_{t(X)}$ is injective. Since  the group $X/F$ is finite, we conclude that  $t(X)$ is finite as well. By the well-known
 Baer theorem, $X = t(X) \oplus Y$, where $Y$ is a torsion-free subgroup of $X$. Let $m= |h(Y)|\leq |X/F|< \infty$ and let $f: G \to G$ be the homomorphism defined by $f(x) =mx$ for every $x\in G$. Then $f(Y) \leq F$ since $h(Y)\cong Y/Y\cap F$ has exponent $m$ (so $mY\subseteq F$). Moreover, $f\restriction_{Y}$ is injective since $Y$ is torsion-free. Therefore, $Y \cong f(Y)=mY$ is free as (isomorphic to) a subgroup of the free group $F$. Hence $Y\cong \bigoplus_\kappa \Z$, so  $\widehat{Y}\cong \T^\kappa$. Therefore, $G \cong \widehat{X }\cong  t(X)\times  \T^\kappa$.  Suppose that $\kappa\ge\omega$. Since $\T[2]^\kappa$ is a zero-dimensional subgroup of $ \T^\kappa$ and $\T[2]^\kappa$ contains an infinite compact metric (zero-dimensional)  subgroup, it follows that $G$ fails property \Zcm, in contradiction with our assumption. This contradiction shows that $\kappa < \infty$. Clearly, $\kappa=\dim G$.
\end{proof}

\begin{corollary}
\label{locally:compact:abelian:group:without:infinite:totally:disconnected:compact:subgroups}
If $G$ is a locally compact abelian group satisfying \Zcm, then $G \cong \R^m\times \T^n \times D$ for some $m, n\in \N$ and some discrete abelian group $D$; in particular, $G$ is a Lie group.
\end{corollary}

\begin{proof} It follows from the structure theory of locally compact abelian groups that $G = \R^n \times G_0$, where $G_0$ has an open
compact subgroup $K$. From the assumption of our corollary and Lemma \ref{exotic:tori:lemma} it follows that $K \cong \T^n
\times F$ for some $n\in \N$ and some finite abelian group $F$. Thus $\T^n$ is an open subgroup of $G_0$. Since $\T^n$ is divisible, one
can write $ G _0 = \T^n \times D$, for an appropriate subgroup $D$ of $G_0$. By the openness of $T^n$ in the open subgroup $K$ of
$G_0$, we deduce that $D$ is a discrete subgroup of $G_0$. Therefore, $G \cong \R^m\times \T^n \times D$ for some $m$.
\end{proof}

In order to extend Corollary \ref{locally:compact:abelian:group:without:infinite:totally:disconnected:compact:subgroups} to arbitrary (not necessarily abelian) compact groups,
we apply Theorem \ref{compact:connected:group:with:Lie:center}(d). 

\begin{lemma}
\label{compact:connected:tds:is:Lie}
\label{compact:Zc:are:Lie}
A compact group $G$ satisfying \Zcm\ is a Lie group.
\end{lemma}

\begin{proof} 
We first consider the special case of our lemma.
\begin{claim}
\label{compact:connected:tds:is:Lie:claim}
The conclusion of our lemma holds when $G$ is connected.
\end{claim}

\begin{proof}
Indeed, the  center $Z(G)$ of $G$ is a compact abelian group satisfying \Zcm, so $Z(G)$ is a Lie group by Lemma \ref{exotic:tori:lemma}.  Assume that $G$ itself is not a Lie group. Applying Theorem \ref{compact:connected:group:with:Lie:center}(d), we conclude that $G$ contains a subgroup $R$ topologically isomorphic to an infinite product of compact connected Lie groups $L_j$.  Since each group $L_j$ has a non-trivial torsion element $x_j$, the subgroup $R$ (hence, also $G$) contains a subgroup $P$ topologically isomorphic to an infinite product of finite non-trivial groups. Since $P$ is zero-dimensional, this contradicts our assumption that $G$ satisfies \Zcm.
\end{proof}

Since $c(G)$ is closed in $G$, it is compact. Being a subgroup of $G$, the group $c(G)$ also satisfies \Zcm.  Claim \ref{compact:connected:tds:is:Lie:claim} implies that $c(G)$ is a Lie group. According to Lee's theorem \cite[Theorem 9.41]{HM}, $G$ has zero-dimensional compact subgroup $N$ such that $G = c(G)\cdot N$.  Since $G$ satisfies \Zcm, $N$ is finite, and so $c(G)$ is a clopen subgroup of $G$.  Since $c(G)$ is locally Euclidean, so is $G$. Thus, $G$ is a Lie group as well.
\end{proof}

\begin{proposition}
\label{locally:compact:are:Lie}
A locally compact group satisfying \Zcm\ is a Lie group.
\end{proposition}

\begin{proof}
By a theorem of Davis \cite{Dav} (see also \cite{CM}), $G$ is {\em homeomorphic} to a product $K \times \R^n \times D$, where $K$ is a compact subgroup of $G$, $n\in \N$ and $D$ is a discrete space. Being a subgroup of $G$, the group $K$ satisfies \Zcm.  So $K$ must be a Lie group by  Lemma \ref{compact:Zc:are:Lie}.
Since $K$ is a locally Euclidean space, the whole group $G \approx K \times \R^n \times D$ is a locally Euclidean space, i.e., $G$ is a Lie group.
\end{proof}

\begin{lemma}
\label{omega-bounded:lemma}
Let $G$ be a topological group such that the closure of every countable subgroup of $G$ is a compact Lie group. Then $G$ is a compact Lie group.
\end{lemma}

\begin{proof}
Let $\mathscr{K}$ be the family of all separable compact subgroups of $G$. The following claim is immediate corollary of the assumption of our lemma.

\begin{claim}\label{Cl1}
\label{members:of:K:are:Lie:groups}\label{c(K):has:finite:index:in:K}
Each $K\in\mathscr{K}$ is a compact Lie group. In particular,  for every $K\in\mathscr{K}$  the subgroup $c(K)$ is open of finite index, so $c(K) \in \mathscr{K}$ and $\dim K = \dim c(K) <\infty$.
\end{claim}

\begin{proof}
The first assertion is just the definition of $ \mathscr{K}$.  The second assertion follows from the fact that the subgroup $c(K)$, as well as the quotient $K/c(K)$, are compact Lie groups. So  $K/c(K)$ is finite and $c(K)$ is a clopen subgroup of $K$ and $\dim K = \dim c(K) <\infty$. Since $K$ is separable, so is $c(K)$. Therefore, $c(K)\in\mathscr{K}$. 
\end{proof}

\begin{claim}
\label{countably:directed}
If $\mathscr{L}$ is a countable subfamily of $\mathscr{K}$, then there exists $K\in\mathscr{K}$ such that $L\subseteq K$ for every $L\in\mathscr{L}$.
\end{claim}

\begin{proof}
For every $L\in\mathscr{L}$ choose a countable dense subset $D_L$ of $L$. Since $\mathscr{L}$ is countable, the subgroup $D$ of $G$ generated by
$\bigcup\{D_L:L\in\mathscr{L}\}$ is countable, so $K=\overline{D}$ is a separable compact group,  so $K\in\mathscr{K}$. If $L\in\mathscr{L}$, then $D_L\subseteq D$, which yields $L=\overline{D_L}\subseteq \overline{D}=K$.
\end{proof}

\begin{claim}
\label{dimension:is:bounded}
There exists $n\in\N$ such that $\dim K\le n$ for every $K\in\mathscr{K}$.
\end{claim}

\begin{proof} Assume the contrary. Then for every $n\in\N$ there exists $L_n\in\mathscr{K}$ such that $\dim L_n\ge n$. Apply Claim~\ref{countably:directed} to the family
$\mathscr{L}=\{L_n:n\in\N\}$ to get $K\in\mathscr{K}$ such that $L_n\subseteq K$ for every $n\in\N$. Clearly, $\dim K\ge \dim L_n\ge n$ for every $n\in\N$. On the other hand, $\dim K$ must be finite by Claim \ref{members:of:K:are:Lie:groups}, a contradiction.
\end{proof}

\begin{claim}
\label{totally:disconnected:subgroups:of:G:are:finite}
Every zero-dimensional compact subgroup $N$ of $G$ is finite, i.e., $G$ satisfies \Zgen.
\end{claim}

\begin{proof}
Let $N$ be a zero-dimensional compact subgroup of $G$ and let $D$ be a countable subgroup of $N$. Then $K=\overline{D}\subseteq N$, so $K$ is zero-dimensional. Since $K\in\mathcal{K}$, $K$ is a compact Lie group by Claim \ref{members:of:K:are:Lie:groups}. It follows that $K$ is  finite. Since $D\subseteq K$, we conclude that $D$ must be finite as well. This proves that $N$ is finite. 
\end{proof}

Use Claims \ref{Cl1} and \ref{dimension:is:bounded} to choose a connected $H\in\mathscr{K}$ such that
\begin{equation}
\label{H:has:maximal:dimension}
\dim H=\max\{\dim K: K\in\mathscr{C}\}.
\end{equation}

\begin{claim}
\label{H:has:finite:index:in:G} $H$ has finite index in $G$.
\end{claim}

\begin{proof}
Assume that $H$ has infinite index in $G$ and choose a countably infinite subset $X$ of $G$ such that
\begin{equation}
\label{eq:8}
X\setminus (F\cdot H)\not=\emptyset
\mbox{ for every finite subset }
F
\mbox{ of }
G.
\end{equation}
Let $D$ be the closed subgroup of $G$ generated by $X$. Then ${D}\in\mathscr{K}$. Since $H\in \mathscr{K}$, applying Claim \ref{countably:directed}
we can choose $K\in\mathscr{K}$ such that $H\subseteq K$ and $X\subseteq D\subseteq K$. Since $H$ is a connected subgroup of $K$, we have $H\subseteq c(K)$. By Claim \ref{c(K):has:finite:index:in:K}, 
$c(K)$ is a clopen subgroup of $K$ and $c(K)\in\mathscr{K}$. From \eqref{H:has:maximal:dimension} we conclude that $\dim c(K)\le \dim H$. Since $H\subseteq c(K)$, the converse inequality $\dim H\le \dim c(K)$ holds as well. Both $H$ and $c(K)$ are compact connected Lie groups by Claim \ref{members:of:K:are:Lie:groups}. 
Since, $H\subseteq c(K)$ and $\dim H=\dim c(K)$, we conclude that $H=c(K)$. Thus, $H$ has finite index in $K$. That is, $X\subseteq K=F\cdot H$ for some finite set $F\subseteq K$, in contradiction with \eqref{eq:8}.
\end{proof}

Since $H$ is a compact Lie group by Claim \ref{members:of:K:are:Lie:groups}, and $H$ has finite index in $G$ by Claim \ref{H:has:finite:index:in:G}, it follows that $G$ is a compact Lie group as well.
\end{proof}

\begin{corollary}
\label{omega:bounded:corollary}
An $\omega$-bounded group satisfying \Zcm\ is a Lie group.
\end{corollary}

\begin{proof}
It suffices to check that the assumption of Lemma \ref{omega-bounded:lemma} is satisfied. Indeed, let $H$ be a countable subgroup of $G$. Since $G$ is $\omega$-bounded, the closure $\overline{H}$ of $H$ in $G$ is compact. By  Lemma \ref{compact:Zc:are:Lie} applied to $\overline{H}$, we conclude that $\overline{H}$ is a compact Lie group.
\end{proof}

Example \ref{cc:example:distinguishing:Zs}(ii) shows that ``$\omega$-bounded'' cannot be weakened to ``countably compact'' in this corollary.

\begin{lemma}
\label{locally:omega:bounded:lemma}
A locally $\omega$-bounded group satisfying \Zcm\ is a Lie group.
\end{lemma}

\begin{proof}
Let $G$ be a locally $\omega$-bounded group satisfying \Zcm.  Fix an open neighbourhood  $U$ of the identity element $1$ of
$G$ such that $\overline{U}$ is $\omega$-bounded. Let $H$ be the subgroup of $G$ generated by $\overline{U}$. Then $H$ is a clopen subgroup of $G$, and $G$ is covered by disjoint translates of $H$, so it suffices to show that $H$ is a Lie group.  Since $H$ is closed in $G$, it satisfies \Zcm.

Take a closed $G_\delta$-subgroup $N$ of $H$ contained in $U$. (Such a subgroup can be obtained by a standard closing off argument.) 
Since $H$ is topologically generated by its $\omega$-bounded subset $\overline{U}$, it is $\aleph_0$-bounded in the sense of Guran \cite{Guran}; that is, for every open neighbourhood $V$ of $H$ there exists a countable set $X\subseteq H$ such that $H=XV$; see \cite{Arh}. As a closed subspace of the $\omega$-bounded space $\overline{U}$, $N$  is $\omega$-bounded. By Corollary \ref{omega:bounded:corollary}, $N$ is a Lie group, so it is metrizable. Since $N$ is a $G_\delta$-set in $H$ and $\{1\}$ is a $G_\delta$-set in $N$, we conclude that $\{1\}$ is a $G_\delta$-subset of $H$.  Since $H$ is $\aleph_0$-bounded group such that $\{1\}$ is a $G_\delta$-subset of $H$, applying the main result in \cite{Arh} we conclude that $H$ has a weaker separable metric topology. Therefore, the subspace $\overline{U}$ of $H$ also has a weaker metric topology.

Now we use a well-known folklore fact that every continuous one-to-one map from a pseudocompact Tychonoff space onto a metric space
is a homeomorphism. Since $\overline{U}$ is $\omega$-bounded (and thus, pseudocompact), this allows us to conclude that $\overline{U}$ is metrizable. Being also pseudocompact, $\overline{U}$ is compact.  Therefore, $H$ is locally compact.  Since $H$ satisfies \Zcm, it is a Lie group by Proposition \ref{locally:compact:are:Lie}.
\end{proof}

\noindent {\bf Proof of Theorem \ref{thm:AB}: }
The implications (i)$\to$(ii)$\to$(iii)$\to$(iv) are established in Proposition \ref{arrows:between:Z},
and the implication  (iv)$\to$(i) follows from Lemma \ref{locally:omega:bounded:lemma}.
\qed

\section{Proofs of Theorems  \ref{thm:NEW} and \ref{last:theorem}}
\label{Min}

Recall that a subgroup $G$ of topological group $K$ is called \emph{locally essential\/} in $K$ if there exists a neighborhood $U$ of the identity of $K$
such that $G$ non-trivially intersects every non-trivial closed normal subgroup of $K$ contained in $U$ \cite{ACDD2}.

We shall need the following local minimality criterion generalizing the classical minimality criterion \cite{DPS,P,S}.
\begin{fact}
\label{minimality-criterion}
 \cite{ACDD2}
A dense subgroup $G$ of a topological group $K$ is locally minimal if and only if $K$ is locally minimal and  $G$ is locally essential in $K$.
\end{fact}

\begin{lemma}
\label{claim:8}
\label{Z:comp:extends:to:completion}
A locally minimal, locally precompact  abelian group satisfying \Zm\ is a Lie group.
\end{lemma}

\begin{proof}  Let $G$ be a locally minimal, locally precompact  abelian group satisfying \Zm. 
Since $G$ is locally minimal, it is locally essential in its completion $K$ by Fact \ref{minimality-criterion}.
Therefore, we can fix a neighborhood $U$ of $0$ in $K$ such that $G$ non-trivially intersects every closed non-trivial subgroup of $K$ contained in $U$.

Since $G$ is locally precompact, $K$ is locally compact.    By Corollary \ref{subgroups:of:Lie:having:Zm}, it suffices to prove that 
$K$ is a Lie group. Suppose that $K$ is not a Lie group.  Since Theorem \ref{thm:AB} has already been proved in Section \ref{O-bou}, we can use its Corollary \ref{loc:compact:corollary} to find an infinite zero-dimensional compact metric subgroup $M$ of $K$.
Since $M$ is a zero-dimensional compact metric group, it has a countable base at $0$ formed by clopen subgroups of $M$. Thus, we can  choose 
a sequence $\{H_n:n\in\N\}$ of clopen subgroups of $M$ contained in $U\cap M$ which form a base at $0$ in $M$. 

Let $n\in\N$. Since $M$ is non-discrete, $H_n$ is non-trivial. Since $H_n$ is closed in $M$, it is compact. In particular, $H_n$ is closed in $K$. 
Since $H_n\subseteq U$, there exists $g_n\in G\cap H_n$ with $g_n\not=0$. 

Clearly, $\{g_n:n\in\N\}$ is an infinite sequence in  $N=G\cap M$ converging to $0$; in particular, $N$  is an infinite subgroup of $G$.  Since $M$ is metric and zero-dimensional, so is $N$. Since $M$ is closed in $K$, $N=G\cap M$ is closed in $G$. We found an infinite closed zero-dimensional subgroup $N$ of $G$. This contradicts our assumption that $G$ satisfies \Zm.
\end{proof}

\noindent{\bf Proof of Theorem \ref{thm:NEW}:}
The implications (i)$\to$(ii)$\to$(iii) follow from Proposition 
\ref{arrows:between:Z}, 
while the implication
(iii)$\to$(i) follows from Lemma \ref{Z:comp:extends:to:completion}.
The final part of our theorem follows from Proposition \ref{locally:precompact:sequentially:complete}.
\qed

\medskip

The following lemma is part of folklore. Its proof is included for completeness only. 

\begin{lemma}\label{T:is:direct:summand}
Let $\T^d$ be a subgroup of a compact abelian group $A$ for some $d\in\N$. Then there is a closed subgroup $B$ of $A$ such that $A=\T^d\times B$. 
\end{lemma}
\begin{proof}
Let $N$ be the annihilator of $\T^d$ in the Pontryagin dual $\widehat{A}$ of $A$.  By a well-known fact (\cite[Proposition 3.4.14 (a)]{DPS}) we have $\widehat{\T^d}=\widehat{A}/N$. Since $\widehat{\T^d}=\Z^d$, the discrete group $\widehat{A}/N$ is a free abelian group. Therefore, $N$ is a direct summand of $\widehat{A}$ (see \cite[Theorem 14.4]{Fu} for this elementary fact). In other words, $\widehat{A}=\Z^d\times N$. Taking the duals in the last equality shows, that $B=\widehat{N}$ is as needed.
\end{proof}

\begin{lemma}
\label{getting:Zm:violating:subgroups:in:G}
Suppose that a topological group $K$ contains an infinite product $R=\prod_{i\in I} L_i$ of
compact Lie groups $L_i$.
If $G$ is a subgroup of $K$ 
satisfying \Zm, then $G\cap L_i=\{e_{L_i}\}$ for all but finitely many $i\in I$.
 \end{lemma}
\begin{proof}
Suppose that $G\cap L_i\not=\{e_{L_i}\}$ for infinitely many $i\in I$.
By trimming the set $I$, we  may assume without loss of generality that 
$G\cap  L_i\not=\{e_{L_i}\}$ for every $i\in I$.

Let $i\in I$ be arbitrary. By our assumption, we can choose $g_i\in G_i$ with $g_i\not=e_{L_i}$.
Let $H_i$ be the closure in $K$ of the cyclic subgroup $\langle g_i\rangle$ of $K$ generated by $g_i$.
Since $g_i\in G\cap L_i\subseteq L_i$ and $L_i$ is a closed subgroup of $K$, 
it follows that $H_i\subseteq L_i$. 
Since $H_i$ is the closure of a cyclic group, $H_i$ is abelian.
Being a closed subgroup of the Lie group $L_i$, $H_i$ is a Lie group.
Since $H_i$ is a closed subgroup of $K$, $G_i=G\cap H_i$ is a closed subgroup of $G$. 
Since $G$ satisfies \Zcm, so does its closed subgroup $G_i$. 
Applying Corollary \ref{subgroups:of:Lie:having:Zm}, we deduce that 
$G_i$ is closed in $H_i$. Since $H_i$ is closed in $K$, so is $G_i$.
From this and $G_i\subseteq L_i$, we conclude that $G_i$ is closed in $L_i$.
Since the latter group is a compact Lie group, so is $G_i$.  
We proved that $G_i=G\cap L_i$ is a compact Lie group.
Since $g_i\in G_i$ and $g_i\not= e_{L_i}$, $G_i$ is a non-trivial compact Lie group.
In particular, $G_i$ contains a non-trivial torsion element $x_i$.
Let $C_i$ be the finite cyclic subgroup of $G_i$ generated by $x_i$.  

Clearly, $P=\prod_{i\in I} C_i$ is a compact metric zero-dimensional subgroup of 
$\prod_{i\in I} L_i=R$, as this product carries the Tychonoff product topology.
Since $R$ is a subgroup of $K$, it follows that 
$N=P\cap G$ is a closed metric zero-dimensional subgroup of $G$.
Note that $C=\bigoplus_{i\in I} C_i\subseteq P\cap G=N$.
Since $C$ is dense in a non-discrete group $P$, $C$ is non-discrete.
From this and $C\subseteq N$, we conclude that $N$ is non-discrete as well.
We have found a non-discrete closed metric zero-dimensional subgroup $N$ of $G$, in contradiction with our assumption that $G$ satisfies \Zm. 
\end{proof}

\begin{corollary}
\label{completion:has:no:infinite:products}
If the completion $K$ of a locally minimal group $G$ contains a subgroup topologically isomorphic to an infinite product of non-trivial compact Lie normal subgroups of $K$, then 
$G$ does not satisfy \Zm.
\end{corollary}
\begin{proof}
Assume the contrary, and let 
$\{L_j:j\in J\}$ 
be an infinite family of non-trivial compact Lie groups such that each $L_j$ is a normal subgroup of $K$ and $\prod_{j\in J} L_j\cong R$ for some subgroup $R$ of $K$.
Since $G$ is locally minimal, it is locally essential in $K$ by Fact
\ref{minimality-criterion}.
Since $G$ is locally essential in $K$, there exists a  neighborhood $W$ of the identity $1$ in $K$ such that  $G\cap N\not=\{1\}$ for every  non-trivial closed normal subgroup $N$ of $K$ with $N\subseteq W$. There exists a finite set 
$S\subseteq J$
such that $\prod_{i\in I} L_i
\subseteq W\cap R$,
where $I=J\setminus S$.
(Here we use the natural identification of the product $\prod_{i\in I} L_i$ with the subgroup of the whole product $R=\prod_{i\in J} L_i$.)

For every $i\in I$, $L_i$ is non-trivial closed normal subgroup of $K$ contained in $W$, so $G\cap L_i\not=\{e_{L_i}\}$ by the choice of $W$.
Since the set $I$ is infinite, 
$G$ does not have property \Zm\ by
Lemma \ref{getting:Zm:violating:subgroups:in:G}.
This contradicts our assumption.
\end{proof}

\begin{theorem}
\label{thm:completion:is:a:Lie:group}
If $G$ is an almost connected, locally minimal, precompact group satisfying \Zm, then the completion $K$ of $G$ is a compact Lie group.
\end{theorem}

\begin{proof} The proof of this theorem is split into three claims.

\begin{claim}
\label{Z(G):is:a:compact:Lie:group}
$Z(G)$ is a compact Lie group.
\end{claim}

\begin{proof} Note that $Z(G)$ is locally minimal, as a closed central subgroup of a locally minimal group \cite{ACDD}. Furthermore, $Z(G)$  satisfies \Zm,  because this property is preserved by taking closed subgroups.  Now Theorem \ref{thm:C} implies that $Z(G)$ is a Lie group.  Since $G$ is precompact, so is $Z(G)$. Since Lie groups are complete, $Z(G)$ is compact. 
\end{proof}

Since $G$ is precompact, its completion $K$ is compact. Since $G$ is locally minimal, $G$ is locally essential in $K$ by  Fact \ref{minimality-criterion}.

\begin{claim}
\label{Z(K):is:a:Lie:group}
$Z(K)$ is a Lie group. 
\end{claim}

\begin{proof} Since $G$ is dense in $K$, the equality $Z(G)=G\cap Z(K)$ holds; in particular, $Z(G)\subseteq Z(K)$.
Since $Z(G)$ is an abelian compact Lie group by Claim \ref{Z(G):is:a:compact:Lie:group}, we can write $Z(G)=\T^d \times F_0$, 
where $d \in \N$ and $F_0$ is a finite abelian group. Then the compact abelian group $Z(K)$ contains a 
subgroup topologically isomorphic to $\T^d$. 

Applying Lemma \ref{T:is:direct:summand} to $A=Z(K)$, we can find a closed subgroup $B$ of $Z(K)$ such that $Z(K) = \T^d \times B$. Since $\T^d\times \{0\}$ is a (finite-index) subgroup of $Z(G)$, one can write $Z(G) = \T^d \times F$, where $F = B\cap Z(G) \cong Z(G)/\T^d = F_0$ is finite. 

Since $G$ is locally essential in $K$, there exist a neighborhood $U_1$ of 0 in $\T^d$ and  a neighborhood $U_2$ of 0 in $B$ such that  $G$ has non-trivial intersection with every closed non-trivial subgroup of $Z(K)$ contained in $U=U_1 \times U_2$. Since $F$ is finite, we may assume, without loss of generality, that $U_2\cap F = \{0\}$. 

We are going to show that $B$ is NSS.  Indeed, let $V$ be a neighborhood of 0 in $B$ whose closure in $B$ is  contained in $U_2$.  Suppose that $H$ is a non-trivial subgroup contained in $V$. Then the closure $N$ of $H$ in $B$ is a non-trivial closed subgroup of $B$ contained in $U_2$. Therefore, $\{0\}\times N$ is a non-trivial  closed subgroup of $\T^d \times B=Z(K)$ contained in $U$. By our choice of $U$, the intersection  $(\{0\}\times N)\cap G$ must be non-trivial. On the other hand,
$$
(\{0\}\times N)\cap G=(\{0\}\times N)\cap(\T^d\times F) =\{0\}\times (N\cap F)\subseteq \{0\}\times (U_2\cap F)=\{0\}\times\{0\}.
$$
This contradiction shows that $V$ contains no non-trivial subgroups of $B$.  Thus, $B$ is NSS.

Being a compact NSS group, $B$ is a Lie group. Then $Z(K)=\T^d\times B$ is a Lie group as well. 
\end{proof}

\begin{claim}
\label{K:is:Lie:claim}
$K$ is a Lie group.
\end{claim}

\begin{proof} 
Since $K$ contains a dense 
almost connected subgroup $G$, $K$ is also almost connected.
By Claim \ref{Z(K):is:a:Lie:group}, $Z(K)$ is a Lie group. 
If $K$ is a not a Lie group, then
Theorem \ref{almost:connected} allows us to conclude that
$c(K)'$ is topologically isomorphic to an infinite product 
of 
non-trivial closed (thus, compact) normal 
Lie subgroups 
of $K$.
Now 
Corollary \ref{completion:has:no:infinite:products} implies that $G$ does not satisfy property \Zm.
This contradiction shows that $K$ is a Lie group.
\end{proof}

The proof of Theorem \ref{thm:completion:is:a:Lie:group} is complete.
\end{proof}

\medskip
\noindent{\bf Proof of Theorem \ref{last:theorem}:}  The implications 
(i)$\to$(ii)$\to$(iii)
follow from Proposition  \ref{arrows:between:Z}.
According to Proposition \ref{locally:precompact:sequentially:complete}, the conditions \Zm\ and \Zcm\ are equivalent for precompact, sequentially complete groups. Therefore, (iii)$\leftrightarrow$(iv).
To prove the implication (iii)$\to$(i),
assume that $G$ is an almost connected, locally  minimal, precompact sequentially complete group satisfying \Zm.
By Theorem \ref{thm:completion:is:a:Lie:group},
the completion $K$ of $G$ is a compact Lie group.
Since $G$ is sequentially complete, it is sequentially closed in $K$.
Since $K$ is metric, $G$ is closed in $K$. Since $G$ is also dense in $K$, we conclude that $G=K$. Therefore, $G$ is a compact Lie group. 
\hfill $\Box$

\section{Compact-like examples distinguishing Lie, \Zgen, \Zm\ and \Zcm}
\label{compact-like:examples}

In this section we exhibit a series of compact-like examples showing that the arrows in \eqref{Z:implications} are not reversible.

To construct our first example, we shall need the notion of an HFD set. Recall that a subset $G$ of $\T^{\omega_1}$ is called an {\em HFD set\/} (an abbreviation for hereditarily finally dense) provided that for every countably infinite subset $X$ of $G$ one can find an ordinal $\alpha<\omega_1$ such  that $q_\alpha(X)$ is dense in $\T^{\omega\setminus\alpha}$, where $q_\alpha: \T^{\omega_1}\to \T^{\omega_1\setminus\alpha}$ be the natural projection defined by 
$q_\alpha(h)=h\restriction_{\omega_1\setminus\alpha}$ for $h\in \T^{\omega_1}$.

It is well known that every HFD subset of $\T^{\omega_1}$ is hereditarily separable, \cc, connected, locally connected and does not contain any non-trivial convergent sequences. Our next proposition adds one more item to the list of the known properties of HFD subsets of  $\T^{\omega_1}$.

\begin{proposition}
\label{HFD:subgroups:are:Z}
An HFD subgroup $G$ of $\T^{\omega_1}$ contains no infinite closed zero-dimensional subgroups;
in particular, $G$ satisfies \Zgen.
\end{proposition}
\begin{proof}
Suppose that $N$ is an infinite closed zero-dimensional subgroup of $G$. Since $N$ is a closed subgroup of the \cc\ group $G$, it is countably compact as well. Since $G$ is zero-dimensional, it is strongly zero-dimensional;  that is, $\dim N=0$. (This follows from \cite{Sh}).

Let $K$ be the closure of $N$ in $\T^{\omega_1}$. Since $N$ is \cc,  $\dim K=\dim N=0$ by Tkachenko's theorem \cite{Tka2}. In particular, $K$ is zero-dimensional.

Fix a countably infinite subset $X$ of $N$. Since $G$ is an HFD subset of $\T^{\omega_1\setminus\alpha}$, there exists  an ordinal $\alpha<\omega_1$ such that $q_\alpha(X)$ is dense in $\T^{\omega_1\setminus\alpha}$. Since $K$ is a compact group containing $X$, it follows that 
$q_\alpha(K)$ is a compact group containing $q_\alpha(X)$. Since $q_\alpha(K)$ is closed in $\T^{\omega_1\setminus\alpha}$
and $q_\alpha(X)$ is dense in $\T^{\omega_1\setminus\alpha}$,  it follows that $q_\alpha(K)=\T^{\omega_1\setminus\alpha}$. 

Since $K$ is a zero-dimensional compact group, its image $q_\alpha(K)$ under continuous homomorphism $q_\alpha$ is zero-dimensional.
On the other hand, $q_\alpha(K)=\T^{\omega_1\setminus\alpha}$ is connected. This contradiction shows that $G$ contains no infinite closed zero-dimensional subgroup $N$. Therefore, all closed zero-dimensional subgroups of $G$ are finite.
Since finite groups are 
discrete,
$G$ satisfies \Zgen.
\end{proof}

\begin{example}
\label{Tkachenko:example}
{\em Under the Continuum Hypothesis, there exists a non-trivial countably compact, connected, 
locally connected
(hereditarily separable) free
abelian group $G$ without non-trivial convergent sequences such that $\{0\}$ 
is the only closed zero-dimensional subgroup of $G$; in particular, $G$ satisfies \Zgen\ but is not Lie\/}. 
Indeed, Tkachenko \cite{Tka1} gave an example, under the Continuum Hypothesis, of an HFD subgroup $G$ of $\T^{\omega_1}$
algebraically isomorphic to the free abelian group of size $\cont$.
Being an HFD subset of $\T^{\omega_1}$, $G$ is hereditarily separable, \cc, connected,
locally connected
 and does not contain 
non-trivial convergent sequences. 
By Proposition \ref{HFD:subgroups:are:Z}, $G$ has no infinite closed zero-dimensional subgroups. Since $G$ is torsion-free, it does not have non-trivial finite subgroups either.
Therefore, $\{0\}$ 
is the only closed zero-dimensional subgroup of $G$, and so
$G$ trivially satisfies \Zgen.
Since $G$ contains no non-trivial convergent sequences,  $G$ is non-metrizable, and so cannot be a Lie group.
\end{example}

In Corollary \ref{Bohr:proposition} we exhibit a precompact group satisfying \Zm\ but failing \Zgen. In item (i) of our next example we outline a pseudocompact group with the same properties.  Item (ii) of the same example shows that, under additional set-theoretic axioms, one can even strengthen 
pseudocompactness to countable compactness in a counter-example to the implication  \Zm$\to$\Zgen.

\begin{example}
\label{cc:example:distinguishing:Zs}
(i)
Let $G$ be the dense pseudocompact subgroup of $\mathbb{Z}(2)^\cont$ 
without non-trivial convergent sequences constructed in \cite{Sirota}.  Clearly, $G$ is zero-dimensional and non-discrete.
It follows from Corollary \ref{non-discrete:group:without:convergent:sequences} that {\em  $G$ is a pseudocompact abelian group which satisfies \Zm\ but does not satisfy \Zgen\/}.

(ii)
{\em Under Martin's Axiom, there exists a countably compact abelian group which satisfies \Zm\ but does not satisfy \Zgen\/}. Indeed, let $G$ be an infinite 
Boolean \cc \ group without non-trivial convergent sequences built by van Douwen under the assumption of MA \cite{vD}. Since $G$ is a \cc\ group of finite exponent, $G$ is zero-dimensional; see \cite{Dzero}. The rest follows from Corollary \ref{non-discrete:group:without:convergent:sequences}.
\end{example}

This example shows that ``$\omega$-bounded" cannot be weakened to ``\cc'' in 
Theorem \ref{thm:AB} and Corollary \ref{characterizing:by:local:omega-boundedness},
even in the ``global'' version.

\begin{example}
\label{ex:Tk2012}
{\em There exists a pseudocompact abelian group satisfying \Zcm\ which does not satisfy \Zm\/}.
Indeed, let $G$ be a pseudocompact abelian group of cardinality $\cont$ such that $G$ contains an infinite cyclic metrizable subgroup $N$ and all countable subgroups of $G$ are closed; such a group is constructed in \cite[Theorem 2.8]{Tk2012}. Since $N$ is countable, it is a closed zero-dimensional subgroup of $G$.
Since $N$ is metrizable, $G$ does not satisfy \Zm. It follows from 
\cite[Corollary 2.7]{Tk2012} that all 
compact
subgroups of $G$ are finite. Thus, $G$ satisfies \Zcm\ 
by Proposition \ref{trivial:ways}(ii).
\end{example}

Proposition \ref{locally:precompact:sequentially:complete} shows that pseudocompactness of Example \ref{ex:Tk2012} cannot be strengthened to its countable compactness.

All examples constructed so far are abelian. We shall now produce a series of examples which are minimal groups $G$ close to being abelian (actually, they are nilpotent of class two, i.e., $G/Z(G)$ is abelian). In order to do so, 
we shall pass abelian examples constructed so far through the ``minimization machine'' 
developed in\cite[Lemma 5.16]{DM}.
The following lemma is extracted from this reference.

\begin{lemma}
\label{L_X:definition}
Let $X$ be an infinite precompact abelian group 
and let $m\not=1$ be its exponent.
Let $K=\T[m]=\{x\in \T: mx=0\}$. 
The discrete Pontryagin dual $D= \widetilde{X}^{\wedge}$  of $X$ acts on $K \times X$ via automorphisms $(t, x) \mapsto (t + \chi(x), x)$ for $(t,x) \in K \oplus X$ and $\chi \in  D$, and the resulting semi-direct product 
$$
L_X=(K \times X) \leftthreetimes D
$$ 
has
the following properties: 
\begin{itemize}
\item [(i)] $L_X$ is a minimal group;
\item [(ii)] $K \times X$ is an open subgroup of $L_X$,  so $L_X$ is locally precompact;
\item[(iii)] $X$ is a closed subgroup of $L_X$;
\item [(iv)]  $Z(L_X ) = K$ and $L_X/Z(L_X) \cong X \times D$, so $L_X$ is nilpotent of class two;
\item [(v)]  if $X$ is connected, then $m=0$ (hence, $K = \T[0]=\T$) and $c(L_X)= K \times X$.
\end{itemize}
\end{lemma}

The next two lemmas are inspired by and extending \cite[Lemma 5.16]{DM}.
The first lemma lists some of the properties which pass from $X$ to $L_X$ and vice versa.

\begin{lemma}
\label{passing:both:ways}
Let $X$ and $L_X$ be as in Lemma \ref{L_X:definition}. 
Then the following holds:
\begin{itemize}
\item [(i)] $L_X$ is a Lie group precisely when $X$ is a Lie group;
\item[(ii)] $L_X$ satisfies satisfies \Zgen\ if and only if $X$ 
satisfies \Zgen;
\item[(iii)] $L_X$ satisfies satisfies \Zm\ if and only if $X$ satisfies \Zm;
\item[(iv)] $L_X$ has property \Zcm\ if and only if $X$ has the same property;
\item [(v)] $L_X$ is (sequentially) complete precisely when $X$ is (sequentially) complete;
\item [(vi)]  $L_X$ has no non-trivial convergent sequences if and only if $X$ has no non-trivial convergent sequences and  $m>0$;
\item[(vii)] $L_X$ is locally connected if and only if $X$ is locally connected.
\end{itemize}
\end{lemma}
\begin{proof}
(i) Suppose that $X$ is a Lie group. Since $K$ is also a Lie group,
$K\times X$ is a Lie group. Since $K\times X$ is an open subgroup of $L_X$ by 
Lemma \ref{L_X:definition}(ii), it follows that $L_X$ is a Lie group. 
Conversely, suppose that $L_X$ is Lie group.
Since $X$ is a closed subgroup of $L_X$ by Lemma \ref{L_X:definition}(iii), it follows that $X$ is a Lie group.

(ii)--(iv) We shall prove these three items simultaneously.
Since $K \times X$ is an open subgroup of $L_X$ by Lemma \ref{L_X:definition}(ii), it is clear that $L_X$ satisfies  \Zgen \ (\Zm\ or \Zcm, respectively) if and only if $K \times X$ satisfies \Zgen\ (respectively, \Zm\ or \Zcm). Assume that $X$ satisfies \Zgen \ (\Zm\ or \Zcm, respectively) and let $N$ be a closed zero-dimensional (metrizable or compact metrizable, respectively) subgroup of $K \times X$.  Since $K$ is compact, the projection $p: K \times X \to X$ is closed. So $N_1 = p(N)$ is a closed subgroup of $X$.  Moreover, as $N$ is zero-dimensional, $\ker p = N \cap (K  \times  \{0\})$ is  a finite subgroup of $K  \times  \{0\}$. Since $N$ is zero-dimensional, this implies that  $N_1\cong N/\ker p$ is a closed zero-dimensional (metrizable, resp., compact metrizable) subgroup of $X$. Therefore, $N_1$ is finite. This proves that $N$ is finite. Hence, $L_X$ satisfies \Zgen (respectively, \Zm\ or \Zcm).

Conversely, if $L_X$ satisfies \Zgen\ (respectively, \Zm\ or \Zcm), then the closed subgroup $X$ of $L_X$ has the same property.

(v) The proof of (v) is similar to that of (i).

(vi) The proof of (vi) is similar to the proof of (ii)--(iv), after taking into account the fact that $K=\T[0]=\T$ has non-trivial convergent sequences.

(vii) Suppose that $X$ is locally connected. If $m=0$, then $K=\T[0]=\T$ is locally connected. If $m>0$, then $K=\T[m]$ is finite and so locally connected.
Since both $X$ and $K$ are locally connected, so is $K\times X$.
Since $K\times X$ is an open subgroup of $L_X$ by 
Lemma \ref{L_X:definition}(ii), it follows that $L_X$ is locally connected.
Suppose now that $L_X$ is locally connected. Then its open subgroup
$K\times X$ is locally connected. Then $X$ is locally connected as well.
\end{proof}

The second lemma lists some of the properties which pass from $X$ to $L_X$ in a local version.

\begin{lemma}
\label{from:global:to:local}
Let $X$, $K$ and $L_X$ be as in Lemma \ref{L_X:definition}, and let
$\mathfrak P$ be a property of topological groups stable under taking 
direct summands (=products) with $K$. If $X$ has property $\mathfrak P$, then $L_X$ is locally $\mathfrak P$.
In particular, 
\begin{itemize}
\item[(i)] if $X$ is (countably) compact, then $L_X$ is locally (countably) compact;
   \item [(ii)]  if $X$ is $\omega$-bounded, then $L_X$ is locally $\omega$-bounded; 
   \item [(iii)] if  $X$ is pseudocompact, then $L_X$ is locally pseudocompact.
\end{itemize}
\end{lemma}
\begin{proof}
Assume 
that $X$ has property  $\mathfrak P$. Our assumption on  $\mathfrak P$ allows us to conclude that $X \times K$ has property  $\mathfrak P$ as well. 
Since $X \times K$ is an open subgroup of $L_X$
by Lemma \ref{L_X:definition}(ii),
$L_X$ is locally  $\mathfrak P$. 
Finally, it suffices to note that all properties  (compactness, $\omega$-boundedness, countable compactness and pseudocompactness) listed in items (i)--(iii)
satisfy our assumptions on $\mathfrak P$. 
\end{proof}

\begin{example}\label{loc_conn+min+loc_cc+Z-Lie}
{\em Under the Continuum Hypothesis, there exists a minimal, locally countably compact (so locally precompact and sequentially complete), 
locally connected
nilpotent group of class 2 (having no non-trivial convergent sequences) which satisfies \Zgen\ but is not Lie\/}.
Indeed, let $G$ be the group constructed in Example   
\ref{Tkachenko:example}. Then $L_G$ is the desired group.
The group $L_G$ is minimal 
and nilpotent of class 2 by items (i) and (iv) of 
Lemma \ref{L_X:definition}, respectively.
Since $G$ is countably compact, $L_G$ is locally countably compact by Lemma \ref{from:global:to:local}(i). 
The rest of the properties of $L_G$ follow from the corresponding properties of $G$ via items (i), (ii), (vi) and (vii) of Lemma \ref{passing:both:ways}.
\end{example}

 This example is to be compared to Theorem \ref{last:theorem}.

Our next group of examples shows that one cannot omit ``abelian'' in Theorem \ref{thm:C} and its corollaries, or replace it by the slightly weaker property ``nilpotent''.  More precisely, we give an example of a minimal locally precompact (consistently, also locally countably compact) nilpotent sequentially complete group satisfying \Zm\ that does not satisfy \Zgen.

\begin{example}\label{MMDD} 
(i)
Let $G$ be any infinite precompact sequentially complete group satisfying \Zm\ but failing property \Zgen; such a group can be constructed 
by applying  Corollary \ref{Bohr:proposition}.
Then {\em $L_G$ is a  minimal, locally precompact, sequentially complete,  nilpotent group of class 2 that satisfies \Zm\ but does not satisfy \Zgen\/}. 
Indeed, it follows from items (i), (ii) and (iv) of Lemma \ref{L_X:definition} that $L_G$ is a minimal, locally precompact, nilpotent group of class 2. 

(ii)
Let $G$ be the topological group from Example \ref{cc:example:distinguishing:Zs}(i).
Then {\em $L_G$ is a  minimal, locally pseudocompact, 
nilpotent group of class 2 that satisfies \Zm\ but does not satisfy \Zgen\/}.
Indeed, since $G$ is pseudocompact, $L_G$ is locally pseudocompact by Lemma \ref{from:global:to:local}(iii).  The rest of the properties of $L_G$ are derived as in item (i).

(iii)
{\em Under the assumption of MA, there exists a minimal, locally countably compact,  nilpotent group of class 2 that satisfies \Zm\ but does not satisfy \Zgen\/}. Indeed, one can take as $G$ a Boolean countably compact group without non-trivial convergent sequences mentioned in Example \ref{cc:example:distinguishing:Zs}(ii). Then $L_G$ becomes locally \cc\ by Lemma \ref{from:global:to:local}(i), in addition to the rest of the properties from item (i). 
\end{example}

In our next example we relax the strong condition \Zgen \ to the weaker condition \Zm, and local countable compactness will be relaxed to 
local pseudocompactness, in order to obtain an example in ZFC.
This example should be compared with Corollary \ref{corollary:last:theorem}. 

\begin{example} 

 {\em There exists a minimal, locally pseudocompact, locally connected, sequentially complete, nilpotent group of class 2 which satisfies \Zm\ but is not Lie\/}.
Indeed,  according to \cite[Corollary 5.6]{GM}, every abelian group of size $\leq 2^{2^\cont}$ admitting a  pseudocompact group topology, admits also a pseudocompact group topology without non-trivial convergent sequences. It follows that the torus group $\T$ admits a pseudocompact group topology $\tau$ without non-trivial convergent sequences. 
Clearly, $\tau$ is sequentially complete. Since every divisible pseudocompact group is connected \cite{W}, the topology $\tau$ must be connected.
Now we apply the ``minimalization machine'' to $(\T,\tau)$ taken as $G$.
The output $L_G$ is the desired group.
\end{example}

It is unclear if the group in the previous example satisfies property \Zgen. 

\begin{example}
{\em There exists a minimal, locally pseudocompact, nilpotent group of class 2 satisfying \Zcm\ which does not satisfy \Zm\/}.
Indeed, it suffices to take the group $G$ as in Example \ref{ex:Tk2012} and pass it through the ``minimalization machine''.
\end{example}

\section{Open questions}

Theorem \ref{last:theorem} leaves open the following question: 

\begin{question}
Is a connected, locally minimal, locally  precompact, sequentially complete group satisfying \Zcm\ a Lie group? 
\end{question}

We do not know whether one can extend Corollary \ref{subgroups:of:Lie:having:Zm} to the non-abelian case.

\begin{question}
Let $K$ be a Lie group and let $G$ be a subgroup $K$ satisfying \Zm.
\begin{itemize}
\item[(i)] Must $G$ be closed in $K$?
\item[(ii)] Is (i) true  when $K$ is compact? 
\item[(iii)] Is (i) true when $G$ is connected (and $K$ is compact)? 
\end{itemize}
\end{question}

Item (iii) of this question should be compared with Theorem \ref{thm:completion:is:a:Lie:group}.

Let us discuss the implications \Zgen\ $\to  Lie$ and \Zm\ $\to  Lie$ for sequentially complete locally minimal (abelian) groups. 
Recall that locally minimal abelian groups need not be locally precompact \cite{ACDD}. 
Since locally minimal locally precompact sequentially complete abelian groups satisfying \Zcm, are compact Lie groups by Theorem \ref{thm:C}, 
it makes sense to  discuss the question without the hypothesis of local precompactness. 
On the other hand, it easily follows from Fact \ref{minimality-criterion}, that every subgroup of a Lie group is locally minimal.
This is why we leave local minimality in the hypothesis of our next question.  

\begin{question}\label{last:question}
 Is every locally minimal sequentially complete abelian group without infinite closed zero-dimensional (metric) subgroups a Lie group?
\end{question}

The completion of a locally minimal (metrizable) group is again locally minimal (metrizable). Since the conclusion of our main theorems (Lie group) makes
 completeness and metrizability of the group a  necessary property, it is natural to  ask also the following question:
 
\begin{question} When is a locally minimal, complete  metric group satisfying 
\Zgen\ 
a Lie group?
\end{question}

Recall that a
topological group $G$ is said to be {\em totally minimal\/} if all Hausdorff quotients of $G$ are minimal. Totally minimal groups are precisely the topological groups satisfying the open mapping theorem. Clearly, all compact groups are totally
 minimal and all totally minimal groups are minimal. 
 
The next group of questions concerns the invertibility of the implication $Lie \ \to $ \Zgen.  

\begin{question}
\label{que:Zgen:not:Lie} 
\begin{itemize}
\item[(i)] Does there exist a pseudocompact abelian group $G$ which satisfies \Zgen\ but is not metrizable (and thus, is not a Lie group)?
What is the answer if one additionally assumes that $G$ is sequentially complete?
\item[(ii)] Does there exist a ZFC example of a \cc\ abelian group which satisfies  \Zgen\ but is not metrizable (and thus, is not a Lie group)?
\item[(iii)] Must a pseudocompact (totally) minimal (abelian) group satisfying \Zgen\ be a Lie group?
\item[(iv)]  Must a countably compact (totally) minimal group satisfying \Zgen\ be a Lie group?
\end{itemize}
\end{question}

The next group of questions concerns the invertibility of the implication \Zgen\ $\to $ \Zm. 

\begin{question}
\label{que:Zm:not:Zgen} 
\begin{itemize}
\item[(i)]  Does there exist a ZFC example of a \cc\ (abelian) group which satisfies  \Zm\ but does not satisfy \Zgen?
\item[(ii)]  Does there exist a pseudocompact (totally) minimal group which satisfies  \Zm\ but does not satisfy \Zgen?
\item[(iii)] Must a countably compact (totally) minimal group satisfying \Zm\ also satisfy \Zgen?
\end{itemize}
\end{question}

In connection with item (ii) we recall that minimal abelian groups which satisfy  \Zm\ are Lie (so satisfy \Zgen). So a group positively answering this item of Question \ref{que:Zm:not:Zgen} is necessarily non-abelian. 

The
answer to Questions \ref{que:Zgen:not:Lie}(iv) and \ref{que:Zm:not:Zgen}(iii) in the ``minimal'' option  is positive when the group in question is either abelian or almost connected; see Corollary \ref {cc:minimal:is:Lie}.

The following question was raised in \cite[Problem 23]{DS_OPIT}: 
Does every countably compact minimal group contain a non-trivial convergent sequence?
A counter-example to this question would be a \cc\ minimal group that satisfies \Zm\ but is not a Lie group.

\begin{question}\label{ques:conn:component}
Is there a pseudocompact (totally) minimal (abelian) group which satisfies \Zcm\ but fails \Zm?
\end{question}

According the Proposition \ref{locally:precompact:sequentially:complete}, the conditions \Zm\ and \Zcm\ are equivalent for locally precompact and sequentially complete groups. Hence the groups witnessing a positive answer to the last question cannot be sequentially complete.


\begin{thebibliography}{99}

\bibitem{ADG}
F.~Ancel, T.~Dobrowolski and J.~Grabowski, \pt{Closed subgroups in Banach spaces},
\jn{Studia Math.} 109 (1994), 277--290.

\bibitem{Arh} 
A.~V.~Arhangel'skii, \pt{Cardinal invariants of topological groups. Embeddings and condensations},
\jn{Dokl. Akad. Nauk SSSR} 247 (1979) 
779--782;
English transl. in: \jn{Soviet Math. Dokl.} 20 (1979) 
783--787.

\bibitem{B}
H.~Bohr, \pt{Collected Mathematical Works}, \jn{Dansk Matemetitisk Forsening}, (E.~F{\o }elner and B.~Jessen, Eds.) (1952). 

\bibitem{ACDD}
L.~Au\ss enhofer, M.~J.~Chasco, D.~Dikranjan and X.~Dominguez,
\pt{Locally minimal topological groups 1}, \jn{J. Math. Anal. Appl.} 370 (2010)
431--452.

\bibitem{ACDD2}
L.~Au\ss enhofer, M.~J.~Chasco, D.~Dikranjan and  X.~Dominguez, 
\pt{Locally minimal topological groups 2}, \jn{J. Math. Anal. Appl.} 380 
(2011) 
552--570.

\bibitem{CM} 
J.~Cleary and S.~Morris, 
\pt{Topologies on locally compact groups}, \jn{Bull Australian Math. Soc.}  38 (1988) 105--111.

\bibitem{CHR} 
W.~W.~Comfort, K.~H.~Hofmann and D.~Remus, \pt{A survey on topological groups and semigroups}, 
\jn{Recent Progress in General Topology} (M.~Husek and J.~van Mill,
Eds.), North Holland, Amsterdam, (1992) 58--144.

\bibitem{vDan}
D.~van Dantzig, 
\pt{Studien over topologische Algebra}, Dissertation, Amsterdam (1931).

\bibitem{Dav} 
H.~F.~Davis, \pt{A note on Haar Measure}, \jn{Proc. Amer. Math. Soc.} 6 (1955) 318--321.

\bibitem{D} 
D.~Dikranjan, \pt{On a class of finite-dimensional compact abelian groups}, 
\jn{Colloq. Math. Soc. Janos Bolyai} 43 (Topology and Appl., Eger (Hungary)) (1983)
215-231.


\bibitem{Dtor} 
D.~Dikranjan, \pt{Density and total  density  of  the  torsion  part  of  a compact group}, \jn{Rend. Accad. Naz. dei XL,  Memorie di Mat.} 108, Vol.
XIV, fasc. 13 (1990) 235--252.

\bibitem{Dzero} 
D.~Dikranjan, \pt{Zero-dimensionality  of some  pseudocompact groups}, \jn{Proc. Amer. Math. Soc.} 120 
(1994) 1299--1308. 
 
\bibitem{Dcc} 
D.~Dikranjan, \pt{Countably compact groups satisfying the open mapping theorem},
\jn{Topol. Appl.} 98 (1999) 81--129. 

\bibitem{DM} 
D.~Dikranjan and M.~Megrelishvili,
\pt{Minimality conditions in topological groups}, 
in \jn{Recent Progress in General Topology III}, Springer Verlag, Berlin, (2013). 

\bibitem{DP0} 
D.~Dikranjan and I.~Prodanov, \pt{Totally minimal topological groups}, \jn{Annuaire Univ. Sofia, Fac. Math. M\' ec.} 69 (1974/75) 5--11.

\bibitem{DP2} 
D.~Dikranjan and I.~Prodanov, \pt{A  class of compact abelian groups}, \jn{Annuaire Univ. Sofia, Fac. Math. M\' ec.} 70 (1975/76) 191--206.

\bibitem{DPS} 
D.~Dikranjan, I.~Prodanov and L.~Stoyanov, \pt{Topological Groups: Characters, Dualities and Minimal Group Topologies}, {Pure and Applied Mathematics}, vol.
 130, Marcel Dekker  Inc., New York-Basel (1989).

\bibitem{DS_PAMS1992}  
D.~Dikranjan and D.~Shakhmatov,  \pt{Compact-like totally dense subgroups  of  compact groups}, \jn{ Proc. Amer. Math. Soc.} 114 
(1992)  1119--1129. 

\bibitem{DS_memo} 
D.~Dikranjan and D.~Shakhmatov,  \pt{Algebraic structure of pseudocompact  groups}, \jn{Memoirs Amer. Math. Soc.} 133/633 (1998), 1--83. 

\bibitem{DS_OPIT} D.~Dikranjan and D.~Shakhmatov,  \pt{Selected topics from the structure theory of topological groups},  
\jn{Open Problems in Topology 2} 
(E. Perl Ed.)
 Elsevier (2007) 
 389--406. 

\bibitem{DSS}
D.~Dikranjan, D.~Shakhmatov and J.~Sp\v{e}v\'{a}k,  \pt{Productivity of sequences with respect to a given weight function}, \jn{Topol.
 Appl.} 158 (2011) 298--324.

\bibitem{DSS1} D. Dikranjan, D. Shakhmatov and ~Sp\v{e}v\'{a}k, \pt{NSS and TAP properties in topological groups closed to being compact}, preprint arXiv:0909.2381 [math.GN].

\bibitem{DSt} 
D.~Dikranjan and L.~Stoyanov, \pt{Compact groups with totally dense torsion part}, \jn{Baku International Topological Conference (Baku, 1987)},  ``Elm'', Baku, (1989) 231--238.

\bibitem{DT1} 
D.~Dikranjan and M.~Tkachenko,  \pt{Sequential completeness of quotient groups}, \jn{Bulletin of Australian Math. Soc.} 61 (2000) 129--151.

\bibitem{DT2}  
D.~Dikranjan and M.~Tkachenko, \pt{Sequentially complete groups: dimension and minimality}, \jn{Jour. Pure Appl. Algebra}, 157 (2001) 215--239.

\bibitem{DG}
T.~Dobrowolski and J.~Grabowski, \pt{Subgroups of Hilbert spaces},
\jn{Math. Z.} 211 (1992) 657--669.

\bibitem{Do} 
D.~Do\"\i tchinov, \pt{Produits de groupes topologiques minimaux}, \jn{Bull. Sci. Math.} 
97 (1972) 59--64.

\bibitem{vD} 
E.~van Douwen, \pt{The product of two countably compact topological groups}, \jn{Trans. Amer. Math. Soc.} 262 (1980) 
417--427. 

\bibitem{Flor} 
P.~Flor, \pt{Zur Bohr-Konvergenz von Folgen},
\jn{Math. Scand.} 23 (1968) 169--170.

\bibitem{Fu} 
L.~Fuchs, \pt{Infinite Abelian Groups}, Vol. I and II, Academic Press, (1970) and (1973).

\bibitem{GM} 
J.~Galindo and  S.~Macario, \pt{Pseudocompact group topologies with no infinite compact subsets}, \jn{J. Pure Appl. Algebra}  215 (2011) 
655--663. 

\bibitem{Glicksberg} 
I.~Glicksberg, \pt{Uniform boundedness for groups}, \jn{Canad. J. Math.} 14 (1962) 269--276.

\bibitem{G} 
V.M.~Glushkov,  \pt{Structure of locally bicompact groups and Hilbert's fifth problem},  
\jn{Uspekhi Mat. Nauk} 12, no. 2  
(1957)  
3--41.
  (In Russian). 
 
\bibitem{Guran} 
I.~Guran, \pt{Topological groups similar to Lindel\"{o}f groups}, 
\jn{Dokl. Akad. Nauk SSSR} 256 (1981) 
1305--1307 (in Russian);
English transl. in: 
\jn{Soviet Math. Dokl.} 23 (1981) 
173--175.

\bibitem{HHM}
S.~H\'ernandez, K.-H.~Hofmann and S.~Morris, \pt{The weights of closed subgroups of a locally compact group},
\jn{J.~Group Theory} 15 (2012), 613--630.

\bibitem{HR}
E.~Hewitt and K.~A.~Ross, \pt{Abstract Harmonic Analysis I}, Die Gr\"uundlehren der Mathematischen Wissenschaften 115, Springer-Verlag (1963).

\bibitem{HM} 
K.-H.~Hofmann and S.~A.~Morris, \pt{The structure of compact groups. A primer for the student --- a handbook for the expert, Second revised and augmented edition} (de Gruyter Studies in Mathematics, 25), Walter de Gruyter \& Co., Berlin (2006). 

\bibitem{L} 
D.~Lee, \pt{Supplements for the identity component in locally compact groups}, \jn{Math. Z.} 104 (1968) 28--49.  


\bibitem{MP}  
S.~Morris and V.~Pestov, \pt{On Lie groups in varieties of topological groups}, \jn{Colloq. Math.} 78 (1998) 
39--47.

\bibitem{P} 
I.~Prodanov, \pt{Precompact minimal group topologies and $p$-adic numbers}, \jn{Annuaire Univ. Sofia Fac. Math. M\'ec.} 66 (1971/72) 249--266.

\bibitem{Sh} 
D.~Shakhmatov, \pt{On zero-dimensionality of subgroups of locally compact groups}, \jn{Comment. Math. Univ. Carolin.} 32 (1991) 
581--582.

\bibitem{Sp} 
D.~Shakhmatov and J.~Sp\v{e}v\'{a}k, \pt{Group-valued continuous functions with the topology of pointwise convergence},
\jn{Topol. Appl.} 157 (2010), 1518--1540.

\bibitem{Sirota}
S.~M.~Sirota, \pt{The product of topological groups and extremal disconnectedness},
\jn{Math. USSR Sbornik} 8 (1969) 169--180.

\bibitem{S} 
R.~M.~Stephenson~Jr., \pt{Minimal topological groups}, \jn{Math. Ann.} 192 (1971) 193--195.

\bibitem{Tka1}
M.~Tkachenko, \pt{Countably compact and pseudocompact topologies on free abelian groups}, \jn{Izv. Vyssh. Uchebn. Zaved. Mat.} (1990) 
no. 5, 
68--75 (in Russian);  English transl. in: \jn{Soviet Math. (Iz. VUZ)}  34  (1990) 
79--86. 

\bibitem{Tka2} 
M.~Tkachenko, \pt{On dimension of locally pseudocompact groups and their quotients}, 
\jn{Comment. Math. Univ. Carolin.} 31 (1990) 159--166.

\bibitem{Tk2012} 
M.~Tkachenko, \pt{Topological groups in which all countable subgroups are closed},
\jn{Topol. Appl.} 159 (2012) 1806--1814.

\bibitem{V} 
N.~T.~Varopoulos, \pt{A theorem on the continuity of homomorphisms of locally compact groups}, 
\jn{Proc. Cambridge Phil. Soc.} 60 (1964) 449--463.

\bibitem{W} 
H.~Wilcox, \pt{Dense subgroups of compact groups}, \jn{Proc. Amer. Math. Soc.} 28 (1971) 578--580. 

\end{thebibliography}
\end{document}